\documentclass[12pt]{article}
\usepackage{amsmath,amssymb,amscd,amsthm}

\usepackage{hyperref}

\usepackage{graphicx} 

\usepackage{tikz}

\usepackage{subcaption}

\newtheorem{theorem}{Theorem}[section]
\newtheorem{proposition}[theorem]{Proposition}
\newtheorem{lemma}[theorem]{Lemma}
\newtheorem{corollary}[theorem]{Corollary}
\newtheorem{conjecture}[theorem]{Conjecture}

\theoremstyle{definition}
\newtheorem{definition}[theorem]{Definition}
\newtheorem{example}[theorem]{Example}
\newtheorem{remark}[theorem]{Remark}

\begin{document}

\newcommand{\legendre}[2]{\genfrac{(}{)}{}{}{#1}{#2}} 
\newcommand{\seqnum}[1]{\href{http://oeis.org/#1}{\underline{#1}}}

\newcommand{\beq}{\begin{equation}}  
\newcommand{\eeq}{\end{equation}}  
\newcommand{\bea}{\begin{eqnarray}}  
\newcommand{\eea}{\end{eqnarray}}  
\newcommand\la{{\lambda}}   
\newcommand\La{{\Lambda}}   
\newcommand\ka{{\kappa}}   
\newcommand\al{{\alpha}}   
\newcommand\be{{\beta}} 
\newcommand\gam{{\gamma}}     
\newcommand\om{{\omega}}  
\newcommand\tal{{\tilde{\alpha}}}  
\newcommand\tbe{{\tilde{\beta}}}   
\newcommand\tla{{\tilde{\lambda}}}  
\newcommand\tmu{{\tilde{\mu}}}  
\newcommand\si{{\sigma}}  
\newcommand\lax{{\bf L}}    
\newcommand\mma{{\bf M}}    
\newcommand\rd{{\mathrm{d}}}  
\newcommand\ri{{\mathrm{i}}} 

\newcommand\rS{{\mathrm{S}}} 

\newcommand{\F}{{\mathbb F}}
\newcommand{\N}{{\mathbb N}}
\newcommand{\Q}{{\mathbb Q}}
\newcommand{\Z}{{\mathbb Z}}
\newcommand{\C}{{\mathbb C}}
\newcommand{\R}{{\mathbb R}}

\newcommand\tI{{\tilde{\mathcal{I}}}}

\newcommand\SH{{\mathcal{S}}}

\newcommand\Tc{{\mathcal{T}}}

\newcommand\Uc{{\mathcal{U}}}

\newcommand\Pc{{\mathcal{P}}}

\newcommand\tr{{{\mathrm{tr}}\,}}

\title{On a family of sequences related to Chebyshev polynomials}
\author{Andrew N. W. Hone\footnote{School of Mathematics, Statistics \& Actuarial Science, 
University of Kent, UK. Currently on leave in the 
School of Mathematics and Statistics, 
University of New South Wales, Sydney NSW 2052, Australia. },$\,$
L. Edson Jeffery and  Robert G. Selcoe
}
\maketitle

\begin{abstract}
The appearance of primes in  
a family of linear recurrence sequences labelled by a positive integer $n$ is considered. The
terms of each sequence correspond to a particular class of Lehmer numbers, or (viewing them as   polynomials in $n$) 
dilated versions of the so-called Chebyshev polynomials of the fourth kind, also known as airfoil polynomials. 
It is proved that when the value of $n$ is given by a dilated Chebyshev polynomial of the first kind evaluated at a suitable integer, 
either the sequence contains a single prime, or no term is prime. For all other values of $n$, it is conjectured that the 
sequence contains infinitely many primes, whose distribution has analogous properties to the distribution of Mersenne primes 
among the Mersenne numbers. Similar results are obtained for the sequences associated with negative integers $n$, which 
correspond to Chebyshev polynomials of the third kind, and to another family of Lehmer numbers. 
\end{abstract}

\section{Introduction}

Consider the  linear recurrence of second order given by 
\beq \label{srec} 
s_{k+2} - n\, s_{k+1} + s_k =0, 
\eeq 
together with the initial conditions
\beq\label{inits} 
s_0 = 1, \qquad s_1 = n+1. 
\eeq 
For each integer $n$, this generates an integer sequence that begins 
\beq\label{sseq} 
\begin{array}{l}
1,n+1,n^2+n-1, n^3+n^2-2n-1, n^4+n^3-3n^2-2n+1,  
\\ n^5+n^4-4n^3-3n^2+3n+1, \ldots . 
\end{array}
\eeq 
The sequence can also be extended backwards to negative indices $k$, so that in particular 
$s_{-1}=-1=-s_0$, which implies 
that it has the 
symmetry 
\beq\label{sym} 
s_k(n)=-s_{-k-1}(n)
\eeq 
for all $k$. 
In this way we obtain a sequence that we denote 
by $(\,s_k(n)\,)_{k\in\Z}$, where the argument denotes the dependence 
on $n$.

We can also interpret this as a sequence of polynomials  in 
the variable $n$, with the integer sequences being obtained by 
substituting  particular values for the argument. From this point of view, it is apparent from the recursive definition that, 
for each $k\geq 0$,  $s_k(n)$ is a monic polynomial of degree $k$ in $n$ with integer coefficients.  In fact, these  are 
rescaled (or dilated) versions of polynomials that are used to determine the pressure distribution in linear airfoil theory, being given 
by 
\beq\label{airfoil} 
s_k(n) = W_k\left(\frac{n}{2}\right), \qquad W_k (\cos\theta ) = \frac{ \sin\Big( (2k+1)\theta /2 \Big) } {\sin(\theta /2)},
\eeq 
where $W_k$ are known as the Chebyshev polynomials of the fourth kind \cite{mh}, or the airfoil polynomials of the second kind (see \cite{nasa}, where the notation $u_k$ is used in place of $W_k$).   
As a function of $\theta$, the expression on the far right-hand side of (\ref{airfoil}) is known as the Dirichlet kernel in Fourier 
analysis, where it is usually denoted $D_k(\theta)$ \cite{dym}. 
Compared with those of the third and fourth kinds, 
the properties of Chebyshev polynomials of the first and second kinds are much 
better known, 
and in what follows we will make extensive use of connections with the latter two sets of polynomials.  

The 
primary goal of  
this article is to describe the case where $n$ is a positive integer, but before proceeding, we 
consider the sequences obtained for some particular small values of $|n|\leq 2$, which will mostly be excluded from subsequent analysis, but are relevant nevertheless. In the case $n=0$, the sequence 
$(\,s_k(0)\,)$ begins 
\beq\label{modn} 
1,1,-1,-1,\ldots, 
\eeq 
and repeats with period 4; we mention this case because it is equivalent to the sequence $(\,s_k(n) \,\bmod n\,)$. 
When $n=1$ the sequence has period 6, being specified by the six initial terms 
\beq\label{n1}
1,2,1,-1,-2,-1, \ldots, 
\eeq 
and for $n=-1$ the sequence repeats the values 
\beq\label{nm1}1,0,-1 
\eeq 
with period 3.
For $n=2$ the sequence grows linearly with $k$, beginning with 
\beq\label{n2}
1,3,5,7,9,11,\ldots, 
\eeq 
and  consists of the odd integers, that is 
\beq\label{lin}
s_k(2) = 2k+1, 
\eeq 
while for $n=-2$ the sequence has period 2, being given by 
\beq\label{nm2} 
s_k(-2) =(-1)^k. 
\eeq 
For each integer $n\geq 3$ the sequence increases monotonically for $k\geq 0$ and grows exponentially with $k$ (see below 
for details).   

Sequence \seqnum{A269254} in the Online Encyclopedia of Integer Sequences (OEIS) \cite{oeis} records the first appearance of a prime 
term in $(\,s_k(n)\,)$. 

\begin{definition} {\bf (Sequence   \seqnum{A269254}.)} \label{adef} 
For each integer $n\geq 1$, if the sequence of terms $(\,s_k(n)\,)_{k\geq 0}$ with non-negative indices contains 
a prime,  then let $a_n$ be  the smallest value of $k\geq 1$ such that $s_k(n)$ is prime; or otherwise, if there is no such 
term, let $a_n=-1$. 
\end{definition} 

There is also sequence  \seqnum{A269253}, whose $n$th term is given by the first prime to 
appear in  $(\,s_k(n)\,)_{k\geq 0}$, or by $-1$ if no prime appears. 

To illustrate the above definition, 
let us start with $n=1$: since  the first prime term in 
the sequence (\ref{n1}) is  $s_1(1)=2$, it follows that $a_1=1$. Similarly, for $n=2$, the first prime in (\ref{n2}) is $s_1(2)=3$, so $a_2=1$; but for $n=3$, the sequence $(\,s_k(3)\,)$ begins $1,4,11,\ldots$, so $a_3=2$. In  cases where a prime term has appeared in the sequence $(\,s_k(n)\,)$, the value of $a_n$ is immediately determined. The sequence 
$(a_n)_{n\geq 1}$ begins with the following terms for $1\leq n\leq 34:$ 
\beq\label{anseq} 
\begin{array}{l}
1, 1, 2, 1, 2, 1, -1, 2, 2, 1, 2, 1, 2, -1, 2, 1, 3, 1, 2, 2, 2, 1, -1, 2, 6, 
2, 3, 1, \\ 3,  1, 
 2, 9, 9, -1,\ldots .
\end{array}
\eeq 
All of the positive values above can be checked very rapidly, and it turns out that all values of $a_n>0$ are of the form $(p-1)/2$, where $p$ is an odd prime: this is a direct consequence of  Lemma \ref{2kp1} below. What is less easy to verify is the negative values $a_{7}=a_{14}=a_{23}=a_{34}=-1$ displayed above, indicating no primes. For instance, when $n=7$, the sequence 
$(\,s_k(7)\,)$ begins with 
\beq\label{n7} 
1,8,55, 377,2584,17711,121393,832040,5702887,\ldots, 
\eeq 
and it can be verified that none of these first few terms are prime; 
but to show that $a_7=-1$ it is necessary to prove that 
$s_k(7)$ is composite for all $k>0$: a proof of this fact can be found in section \ref{fact}, while another proof appears in section \ref{chebv} in a broader setting.  

In fact, in order to understand the family of sequences $(\,s_k(n)\,)$ with positive $n$, it will be natural 
to consider negative integer values of $n$ as well. In that case, it is helpful to define the family of sequences   $(\,r_k(n)\,)$ 
given by 
\beq\label{rsrel} 
r_k(n)=(-1)^k\, s_k(-n).
\eeq 
It is straightforward to show by induction that, for fixed $n$, the sequence 
$(\,r_k(n)\,)$ satisfies the same recurrence (\ref{srec}) but with different initial conditions, 
namely 
$$ 
r_{k+2} - n\, r_{k+1} + r_k =0, 
$$ 
together with 
\beq\label{rinits} 
r_0 = 1, \qquad r_1 = n-1. 
\eeq 
For integer $n$, 
this generates an integer sequence that begins 
\beq\label{rseq} \begin{array}{l}
1,n-1,n^2-n-1, n^3-n^2-2n+1, n^4-n^3-3n^2+2n+1, \\ n^5-n^4-4n^3+3n^2+3n-1, \ldots . 
\end{array}
\eeq 
Up to rescaling $n$ by a  factor of 2, this sequence of polynomials  arises in describing the downwash distribution 
in linear airfoil theory,
and in this context they are referred to as 
 the airfoil polynomials of the first kind  \cite{nasa}, denoted $t_k$; with the alternative notation 
 $V_k$ they are also referred to as the Chebyshev polynomials of the third kind \cite{mh}, so that 
\beq\label{1stairfoil} 
r_k(n) = V_k\left(\frac{n}{2}\right), \qquad V_k (\cos\theta ) = \frac{ \cos\Big( (2k+1)\theta /2 \Big) } {\cos(\theta /2)}.
\eeq 
Since $n=2\cos\theta$, the identity (\ref{rsrel}) can also be obtained immediately by taking $\theta\to\theta+\pi$ in (\ref{airfoil}), and comparing with (\ref{1stairfoil}). 

There is  another OEIS sequence that is relevant here, corresponding to the first appearance of a prime in the 
sequence defined by (\ref{rsrel}) for each positive integer $n$.  
\begin{definition} {\bf (Sequence  \seqnum{A269252}.)} \label{atdef} 
For each integer $n\geq 1$, if the sequence of terms $(\,r_k(n)\,)_{k\geq 0}$ with non-negative indices contains 
a prime,  then let $\tilde{a}_n$ be  the smallest value of $k\geq 1$ such that $r_k(n)$ is prime; or otherwise, if there is no such 
term, let $\tilde{a}_n=-1$. 
\end{definition} 

There is also sequence  \seqnum{A269251}, whose $n$th term is given by the first prime to 
appear in  $(\,r_k(n)\,)_{k\geq 0}$, or by $-1$ if no prime appears. 

For comparison with 
 \seqnum{A269254}, note that the first few terms of   \seqnum{A269252} for $1\leq n\leq 34$ are given by  
\beq\label{atnseq} \begin{array}{l}
-1, -1, 1, 1, 2,
 1, 2, 1, 2, 2, 
 2, 1, 3, 1, 3, 
 2, 2, 1, 14, 1, 
 2, 2, 3, 1, 2, 
 5, 2, 36,  \\ 2, 1, 
 2, 1, 15, -1, 
\ldots .
\end{array}
\eeq 
The two initial $-1$ values that appear above for $n=1,2$ clearly correspond to 
(\ref{nm1}) and (\ref{nm2}), respectively, while the first non-trivial case 
to consider is the value $-1$ that appears for $n=34$, corresponding to the sequence 
$(\, r_k(34)\,)$, which begins   with 
\beq\label{nm34} 
    1,
                               33,
                              1121,
                             38081,
                            1293633,
                            43945441,
                           1492851361,
                          50713000833
,\ldots;
\eeq 
again it can be verified that none of these first few terms are prime, while a proof 
that all terms with $k>0$ are composite for  this and certain other values of $n$ 
is given in section 5. 

Aside from the connection with Chebyshev polynomials, the numbers 
$s_k(n)$ and $r_k(n)$ also correspond to particular instances of Lehmer numbers with 
odd index, which are closely related to sequences of Lucas numbers.  
Prime divisors 
in sequences of Lucas and 
Lehmer numbers have  been studied for some time; see 
e.g.\ \cite{bilu, rotk, schinzel, stewart} for some 
general results, or see   
\cite{estw} for a more 
elementary introduction to primitive divisors. However, to the best of our 
knowledge, the question of when such sequences are 
without prime terms,  or of where the first prime appears in such 
sequences,  has not been considered in detail before, 
except in the case $n=6$.

The case $n=6$ corresponds to the so-called NSW numbers, 
named after \cite{nsw} (sequence \seqnum{A002315}). An NSW number $q$ can be characterized 
by 
there being some $r$ such that the pair of positive integers $(q,r)$  
satisfies the Diophantine equation 
$$
q^2+1=2r^2.
$$
The sequence of NSW numbers   is given by $q=s_k(6)$ for $k\geq 0$, 
with the corresponding solution to the above equation 
being $(q,r)=( s_k(6),r_k(6))$. The subsequence of prime NSW numbers 
is of particular interest in relation to finite simple groups of square order: 
the symplectic group of dimension 4 over the finite field $\F_q$ has a 
square order if and only if $q$ is a prime NSW number, with the order 
being $(q^2(q^2-1)r)^2$. The first prime NSW number is $s_1(6)=7$, 
and the symplectic group of dimension 4 over $\F_7$ is of order 
$11760^2$.

For an arbitrary  linear recurrence relation of second order, that is 
$$
x_{k+2}=a\, x_{k+1}+b\,  x_k, \qquad (a,b)\in\Z^2, 
$$ 
the general question of whether it generates a sequence 
without prime terms has been considered for some time. If either the coefficients 
$a,b$ or the two initial values $x_0,x_1$ have a common factor then it is obvious 
that all terms  $x_k$ for $k\geq2$ have the same common factor, so 
the main case of interest is where $\gcd(a,b)=1=\gcd(x_0,x_1)$. 
In the case of the  Fibonacci recurrence with $a=b=1$,  
the groundbreaking result was due to 
Graham, who found a sequence 
whose first two terms are relatively prime and which consists only of composite integers \cite{graham}.  
This result was generalized to arbitrary 
second-order recurrences by Somer \cite{somer} and Dubickas et al.\ \cite{composite}.

An outline of the paper is as follows. The next section serves to set up notation 
and provide a rapid introduction 
to the properties of dilated Chebyshev polynomials of the first and second kinds, 
which will be used extensively in the sequel, and also contains  the 
required definitions of the corresponding 
sequences of Lucas and Lehmer numbers that appear subsequently.  Section 3 provides a 
very brief review of some standard facts about linear recurrence sequences and products 
of such sequences, before a presentation of examples and preliminary results about values of $n$ for which the terms  
$s_k(n)$ factor into a product of two linear recurrence sequences; this serves  
to illustrate and motivate 
the results which appear in section 5. As preparation for the latter, 
section 4 contains a collection of various general properties  
of the sequences 
$(\,s_k(n)\,)$. The main results of the paper, 
on the factorization of $(\,s_k(n)\,)$ and $(\,r_k(n)\,)$ when  
$n$ is a dilated Chebyshev polynomial of the first kind evaluated at 
integer argument (Chebyshev values), 
are presented in section 5. Section 6 considers the appearance 
of primes in these sequences in the case that $n$ is  not one of the Chebyshev values, 
and gives heuristic arguments and numerical evidence to support a 
conjecture to the effect that the 
behaviour is analogous to that of the sequence of Mersenne primes. 
Some conclusions are made in the final section, and there are two appendices:  
the first is a collection of data on prime appearances, and the second 
is a brief catalogue of related 
sequences in the OEIS.  

This paper arose out of a series of posts to the {\tt SeqFan} mailing list, with contributions from many people, 
both professional and recreational mathematicians. 
Our aim throughout has been to make the presentation  as explicit as possible, and for the sake of completeness 
we have stated several standard facts and definitions, as well as providing direct,  elementary proofs of almost every statement (even when some of them are  particular cases of more  general  results 
in the literature). We hope that in this form it  will be possible for our work to be appreciated by sequence 
enthusiasts of every persuasion.

\section{Dilated Chebyshev polynomials and Lehmer numbers} 

The families of Chebyshev polynomials arise in the theory of orthogonal polynomials, and have diverse applications in numerical analysis \cite{mh}. There are 
four such families, 
and while the Chebyshev polynomials  of the first and second kinds are well studied in the literature, those of the third and fourth kinds are not so well known, and some of their 
connections to arithmetical problems have only been considered quite recently \cite{jonny}. 

In order to define scaled versions of the standard Chebsyhev polynomials in terms of trigonometric functions, 
let 
$$ 
n = 2\cos\theta  = \la + \la^{-1}, 
$$ 
so that we may write 
\beq\label{ladef} 
\la = \frac{n+\sqrt{n^2 - 4}}{2} = e^{\ri\theta}, 
\eeq 
where $\ri=\sqrt{-1}$. 
Then the formulae 
\beq\label{1st} 
\Tc_k(2\cos\theta) = 2\, \cos (k\theta), \qquad 
\Uc_k(2\cos\theta) = \frac{\sin\left((k+1)\theta\right)}{\sin\theta} 
\eeq 
define $\Tc_k$, $\Uc_k$ as polynomials in $n$, for all $k\in\Z$. 
In chapter 18 of \cite{dlmf} these polynomials are referred to as the dilated Chebyshev polynomials of the first and second kinds, and they are denoted by $C_k$, $S_k$ respectively. In  standard notation, the classical Chebyshev polynomials of the first and second kinds are written as  $T_k$ and $U_k$, 
and their precise 
relationship with the dilated  polynomials used here   is as follows: 
$$ 
\Tc_k(n) = 2\,T_k\left(\frac{n}{2}\right), \qquad \Uc_k(n) = U_k\left(\frac{n}{2}\right). 
$$ 
 
It is straightforward to show from the definitions (\ref{1st}) that the dilated Chebyshev polynomials of the first and second kinds satisfy the same recurrence (\ref{srec}) as the sequence $(\,s_k(n) \,)$, but with different initial values. For example, to verify that the sequence $(\,\Uc_k(n) \,)$ satisfies the recurrence, it is sufficient to note that 
\beq \label{urec} \begin{array}{rcl}
\Uc_{k}(n) - n \, \Uc_{k-1}(n) +\Uc_{k-2}(n) &  
= &  \frac{\sin\left((k+1)\theta\right)  -2\cos\theta \,{\sin(k\theta)}+ {\sin\left((k-1)\theta\right)} }{\sin\theta}, 
\end{array} 
\eeq 
and then observe that the right-hand side above vanishes as a consequence of  
the addition formula for sine in the form 
\beq\label{sine} 
\sin (\theta+\phi)+\sin(\theta -\phi) = 2\sin\theta \cos\phi. 
\eeq 
For comparison with other texts, we note that the sequence of dilated first kind polynomials begins thus: 
\beq\label{1stk} 
(\, \Tc_k(n)\,): \qquad 2,n, n^2 -2, n^3-3n, n^4 - 4n^2+2, n^5-5n^3+5n, \ldots . 
\eeq 
In contrast, the sequence of dilated second  kind polynomials begins as 
\beq\label{2ndk} 
(\, \Uc_k(n)\,): \qquad 1,n, n^2 -1, n^3-2n, n^4 - 3n^2+1, n^5-4n^3+3n, \ldots . 
\eeq 
For future reference, we note the standard identities 
\beq\label{tid} 
\Tc_{ab}(n) = \Tc_a(\Tc_b(n)) 
\eeq 
and 
\beq\label{uid} 
\Uc_{ab-1}(n) = \Uc_{a-1}(\Tc_b(n)) \,\Uc_{b-1}(n), 
\eeq 
which follow from the trigonometric definitions above. 

In order to get a formula for the coefficients of the polynomials $s_k(n)$,  we present an explicit expansion for 
the dilated Chebyshev polynomials of the second kind. Although this can be found elsewhere in the literature (cf.\ equations (5.74) and (6.129) in \cite{conc}), for 
completeness we present an elementary proof. 

\begin{proposition}\label{uprop} 
The dilated Chebyshev polynomials of the second kind are given by 
\beq\label{uexp} 
\Uc_k(n) = \sum_{i=0}^{\left \lfloor{\frac{k}{2}}\right \rfloor}(-1)^i   { k-i \choose i} 
n^{k-2i}. \eeq  
\end{proposition}
\begin{proof} First note that for $k=0,1$ the sum on the right-hand side of (\ref{uexp}) 
agrees with the initial terms $\Uc_0=1$, $\Uc_1=n$. Then, upon substituting the sum formula into the recurrence (\ref{urec}) 
and comparing powers of $n$, after dividing by $(-1)^i$ we see that  the coefficient of $n^{k-2i}$ yields the  identity 
$$ 
 { k-i \choose i} -  { k-i -1 \choose i} -  { k-i -1 \choose i-1 } = 0 
$$ 
for binomial coefficients. 
 Thus the sequences defined by the left-hand and right-hand sides of (\ref{uexp}) satisfy the same recurrence with the 
same initial conditions, so they must coincide.
\end{proof}    

For use in what follows, we also define Lehmer numbers. Given a quadratic polynomial in $X$  
with roots $\al,\be$,  that is 
\beq\label{lquad} 
X^2 -\sqrt{R}\, X + Q=(X-\al)(X-\be), \qquad Q,R\in \Z, 
\eeq 
where it is assumed that $Q,R$ are coprime and $\al/\be$ is not a root of unity, 
there are two associated sequences of Lehmer numbers, 
which (adapting the 
notation of 
\cite{epsw}) we denote by $L^-_k(\sqrt{R},Q)$ and 
$L^+_k(\sqrt{R},Q)$, where 
\beq\label{lm} 
L^-_k(\sqrt{R},Q) = 
\left\{ 
\begin{aligned} 
 &\frac{\al^k-\be^k}{\al-\be}, && k\,\,\mathrm{odd} \\
&\frac{\al^k-\be^k}{\al^2-\be^2}, && k\,\,\mathrm{even}, 
\end{aligned} 
\right.
\eeq 
and 
\beq\label{lp} 
L^+_k(\sqrt{R},Q) = 
\left\{ 
\begin{aligned} 
 &\frac{\al^k+\be^k}{\al+\be}, && k\,\,\mathrm{odd} \\
& {\al^k+\be^k}, && k\,\,\mathrm{even}. 
\end{aligned} 
\right.
\eeq 
The sequences of  Lehmer numbers can be viewed as generalizations of the Lucas sequences. Assuming that 
$R$ is a perfect square, so $P=\sqrt{R}\in\Z$, the two types of Lucas sequences associated to the 
quadratic $X^2-PX+Q=(X-\al)(X-\be)$ 
are given 
by 
\beq\label{luc}
\ell^-_k (P,Q) = \frac{\al^k-\be^k}{\al - \be}, \qquad \ell^+_k (P,Q)= \al^k +\be^k; 
\eeq 
the corresponding Lehmer numbers $L^{\pm}_k(P,Q)$ are obtained from the Lucas numbers 
$\ell^{\pm}_k(P,Q)$  by removing trivial factors. 

From the above definitions, there is a  
clear link between Chebyshev polynomials and Lucas/Lehmer numbers, which can be summarized in the following

\begin{proposition}\label{lehch} 
For integer values $n$, the sequences of dilated Chebyshev polynomials of the first and second kinds coincide with  
particular Lucas sequences, that is 
\beq\label{lucas} 
\Tc_k(n) = \ell_k^+(n,1), \qquad \Uc_{k-1}(n) = \ell_k^-(n,1), 
\eeq 
while  the  sequences
 generated by (\ref{srec}) with initial values (\ref{rinits}) and (\ref{inits})  consist of Lehmer numbers 
with odd index, namely 
\beq\label{skleh} 
r_k(n) =  L^+_{2k+1}\left(\sqrt{n+2},1\right), \quad 
s_k(n) = L^-_{2k+1}\left(\sqrt{n+2},1\right)
\eeq  
 respectively, 
for all $k$. 
\end{proposition} 
\begin{proof} 
The formulae for $\Tc_k$ and $\Uc_k$ follow immediately from  comparison of (\ref{1st}) with (\ref{luc}), requiring from (\ref{ladef}) that $\al =\la =e^{\ri\theta} =\be^{-1}$ in (\ref{lquad}). For the proof of the second part of the statement, note that 
taking $k=0,1$ gives $L_1^-(\sqrt{n+2},1)=1$ 
and 
$L_3^-(\sqrt{n+2},1)=n+1$, while a short calculation shows that  
$L^-_{2k+1}(\sqrt{n+2},1)$ satisfies the same recurrence  (\ref{srec}) as $s_k(n)$, 
and similarly for the other sequence given by 
$r_k(n)=(-1)^k s_k(-n)$; so for each equation in  (\ref{skleh}), the sequences 
given by their left/right-hand sides coincide. 
\end{proof} 

\begin{remark}  
There are also expressions for 
$r_k(n)$ and $s_k(n)$ 
in terms of 
dilated Chebyshev polynomials of the first/second kinds, respectively,  with 
argument $\sqrt{n+2}$: see  (\ref{p2f}) and (\ref{mainf}) below.  
\end{remark} 

By writing the roots of the polynomial $X^2-\sqrt{n+2}\,X +1$ as 
\beq\label{aldef} 
\al^{\pm 1} = \frac{\sqrt{n+2}\pm\sqrt{n-2}}{2}, 
\eeq 
we have an alternative way to identify  the terms in sequence  \seqnum{A269254}. 
\begin{corollary} {\bf (Alternative characterization of sequence   \seqnum{A269254}.)} \label{altdef}
For each $n\geq 3$, if $\al$ is defined by (\ref{aldef}), then $a_n$ is that positive integer $k$ yielding the smallest prime of
  the form 
\beq\label{alform}
\frac{\al^{2k+1} - \al^{-(2k+1)}}{\al-\al^{-1}}, 
\eeq 
or $a_n=-1$ 
if no such $k$ exists. 
\end{corollary} 

The sequence  \seqnum{A269252} can be identified in terms of the characteristic roots of 
(\ref{aldef}) in a similar way.
\begin{corollary} {\bf (Alternative characterization of sequence   \seqnum{A269252}.)} \label{altrdef}
For each $n\geq 3$, if $\al$ is defined by (\ref{aldef}), then $\tilde{a}_n$ is that positive integer $k$ yielding the smallest prime of
  the form 
\beq\label{ralform}
\frac{\al^{2k+1} + \al^{-(2k+1)}}{\al+\al^{-1}}, 
\eeq 
or $\tilde{a}_n=-1$ 
if no such $k$ exists. 
\end{corollary}

\section{Some surprising factorizations} \label{fact}

In this section we briefly recall some basic facts about sequences generated by linear recurrences, before looking at 
some special properties of the family of sequences $(\,s_k(n)\,)$.  We assume that all recurrences are defined over the field $\C$ of complex numbers.  (In the next  
section we will also consider recurrences in finite fields or residue rings.) 
For a broad review of linear recurrences in a more general setting, 
the reader is referred to \cite{epsw}. 

For what follows, it is convenient to make use of the forward shift, denoted $\rS$, which 
is a linear  operator that acts on any 
sequence $(f_k)$ with index $k$ according to 
$$\rS \, f_k = f_{k+1}.$$  
With this notation, the fact that a sequence $(x_k)$ satisfies a linear recurrence relation of 
order $N$ with constant coefficients can be expressed in the form 
\beq\label{linrel}
F(\rS) \, x_k =0, 
\eeq 
where  $F$ (of degree $N$) is  the characteristic polynomial of the recurrence.

\begin{definition} A {\it decimation} of a sequence $(x_k)_{k\in\Z}$ is any subsequence of the form 
$(x_{i+dk})_{k\in\Z}$, for some fixed integers $i,d$, with $d\geq 2$. A particular name for the  case $d=2$  is a 
{\it bisection}, $d=3$ is  a
{\it trisection}, and in general this is a decimation 
of order $d$.  
\end{definition}   

\begin{remark} The case of decimations of linear recurrences defined over finite fields is considered in \cite{dm}. 
\end{remark}

Since, at least in the case that all the roots $\la_1,\la_2,\ldots, \la_N$ of $F$ are distinct, 
the general solution of  (\ref{linrel}) can be written as a linear combination of  $k$th powers of the $\la_j$, 
it is apparent that 
the terms of a decimation 
of order $d$ are given by 
$
x_{i+dk} = \sum_{j=1}^NA_j\,\la_j^{dk}
$, 
for some coefficients $A_j$. 
Hence  
the  decimation 
 satisfies the linear recurrence  
\beq \label{diss} 
\prod_{j=1}^N (\rS -\la_j^d) \, x_{i+dk} = 0. 
\eeq 
(The recurrence for the decimation  
has the same  form  in the case of repeated roots.) 
Decimations of the sequence $(\,s_k(n)\,)$ will be considered in Proposition \ref{linrels} in the next section. 

Given two sequences $(x_k)$, $(y_k)$ that satisfy linear recurrences of order $N,M$ respectively, the  product 
sequence $$(z_k)=(x_ky_k)$$ also satisfies a linear recurrence. The following result is well known. 
\begin{theorem}\label{prodt} 
The product $(z_k)=(x_ky_k)$ of two sequences that satisfy linear recurrences of order $N,M$ satisfies a 
linear recurrence of order at most $NM$.  
\end{theorem}   
To prove the theorem in the generic situation where the recurrences for $(x_k)$, $(y_k)$ both have 
distinct characteristic roots, given by  
$\la_i$, $1\leq i\leq N$ and  
$\mu_j$, $1\leq j\leq M$ respectively, observe that each product $\nu_{i,j}= \la_i\mu_j$ is a characteristic root for the linear recurrence satisfied by $(z_k) $,  
that is 
\beq\label{nusum}
\prod_{i,j}(\rS - \nu_{i,j}) \, z_k =0, 
\eeq
where the sum is over a maximum set of $i,j$ that give distinct $\nu_{i,j}$; so 
if the $\nu_{i,j}$ are all different from each other 
then the order of the recurrence is exactly $MN$, but the order could be smaller if 
some of the $\nu_{i,j}$ coincide. For the general situation with repeated roots, see \cite{zm}. 

We now consider an observation  concerning the sequences $(\,s_k(n)\,)$  for the special values 
$n=j^2-2$ where $j\in \Z$, which includes the cases $n=7,14,23,34$ that have $a_n=-1$ in (\ref{anseq}). 
The  fact is that for all these values, there is a  
 factorization  of the form 
\beq\label{t2fac} 
s_k(j^2-2) =r_k(j) s_k(j) ,
\eeq
where both factors on the right-hand side above satisfy a 
linear recurrence of second order. This is  surprising, because in the light of Theorem \ref{prodt} one would naively 
expect such a product to satisfy a recurrence of order 4. 
\begin{theorem}\label{t2} 
For the values $n=j^2-2$, the terms of the sequence $(\,s_k(n)\,)$ admit the factorization 
(\ref{t2fac}), 
where $r_k(j)$ satisfies the same recurrence as $s_k(j)$, that is 
\beq\label{p2} 
r_{k+2}(j) -j\, r_{k+1}(j) +r_k(j)=0, 
\eeq 
with the  initial values 
\beq\label{rjinits} 
r_0(j)=1, \qquad r_1(j)=j-1. 
\eeq 
Thus for all $j\in\Z$ the formula (\ref{t2fac}) expresses $s_k(j^2-2)$ as a product of two integers. 
\end{theorem} 
\begin{proof} 
In the case $n=j^2-2$, the formula (\ref{aldef}) fixes the characteristic  roots of the recurrence (\ref{srec}) 
as $\la=\al^2$, $\la^{-1}=\al^{-2}$, where $\al = (j+\sqrt{j^2-4})/2$; so $\al+\al^{-1}=j$, and the square root of 
$\al$ can be fixed so that $\al^{1/2}+\al^{-1/2}=\sqrt{j+2}$. Then, by applying the difference of two squares to the numerator and denominator of (\ref{alform}), it follows that 
\beq\label{dos}  
s_k(j^2-2) 
= \left(\frac{\al^{(2k+1)/2}+\al^{-(2k+1)/2}}{\al^{1/2}+\al^{-1/2}}\right)\, 
\left(\frac{\al^{(2k+1)/2}-\al^{-(2k+1)/2}}{\al^{1/2}-\al^{-1/2}}\right) 
\eeq 
which is the factorization (\ref{t2fac}) with 
$r_k(j)$  
and $s_k(j)$ 
given by making the replacement  $n \to j$ in  
(\ref{skleh}). 
(For an alternative expression for these factors, see (\ref{rsu}) in 
Remark \ref{altf} below.) 
Each of the factors above is a linear combination of $k$th powers of the characteristic roots $\al, \al^{-1}$, and $r_k(j)=(-1)^k s_k(-j)$ as in (\ref{rsrel}), so they each 
satisfy the same recurrence 
(\ref{p2}) with an appropriate set of initial values. 
\end{proof} 

\begin{remark}\label{klee}  
Generically, the product of any two solutions of the recurrence (\ref{p2}) would have 
three characteristic roots, namely  
$\al^2,\al^{-2},1$, giving a recurrence of order 3 in (\ref{nusum}), but the potential root 1 cancels from the product 
(\ref{dos}), giving the second-order  recurrence (\ref{srec}) with $n=j^2-2$. 
An inductive proof of the preceding result was given by Klee in a post to the Seqfan mailing list: see \cite{klee} for 
details. However, the factorization (\ref{dos}) in the form 
$   L^-_{2k+1}(j,1) =  L^+_{2k+1}(\sqrt{j+2},1)\, L^-_{2k+1}(\sqrt{j+2},1)$ appears to be well known in the literature on Lehmer numbers; see e.g.\ \cite{caldwell} 
and references.\footnote{In particular, see  \url{http://primes.utm.edu/top20/page.php?id=47 } 
for a sketch of a proof of Theorem \ref{t2}.}
\end{remark} 
\begin{example}
In the case $n=7$, there is the factorization 
$$ 
s_k(7)=r_k(3)\,s_k(3), 
$$ 
where the first terms of the factor sequences are 
$$ \begin{array}{l}
(\,r_k(3)\,):  1,2,5,13,34,89,233,610,1597,\ldots , \\
(\,s_k(3)\,):  1,4,11,29,76,199,521,1364,3571,\ldots,
\end{array} 
$$ 
which multiply together to give the terms in (\ref{n7}). Since both $(\,r_k(3)\,)$
and $(\,s_k(3)\,)$ are strictly increasing  sequences, it follows that  $s_k(7)$ is composite for 
all $k\geq 1$, and hence $a_7=-1$, as asserted previously.
\end{example}

As a consequence of the factorization (\ref{t2fac}), one can show similarly that for all integers $j\geq 3$, the terms $s_k(j^2-2)$ 
are composite for $k\geq 1$, and thus $a_{j^2-2}=-1$ for all $j\geq 3$ (for full details, see the proof of 
Theorem \ref{maint} below). In particular, Theorem \ref{t2} accounts for all the values $n=7,14,23,34$ with $a_n=-1$ that are shown in the list (\ref{anseq}).

The question is now whether there are other cases with $a_n=-1$, for which $n\neq j^2 -2$ for 
some $j$. 
It turns out that the answer to this question is affirmative, and the first case with $a_n=-1$  that does not fit into the above pattern is  $n=110$ \cite{setal}.  

\begin{example}\label{n110}  
The sequence $(\,s_k(110)\,)_{k\geq 0}$, beginning with  
\beq\label{n110s} \begin{array}{l}
1, 111, 12209, 1342879, 147704481, 16246150031, 1786928798929, \\
196545921732159, \ldots , 
\end{array}
\eeq 
appears as number  \seqnum{A298677} in the OEIS. To see that none of the terms are prime, first of all note that the sequence 
$(\,s_k(110)\, \bmod 111\,)$ is periodic with period 3: it is equivalent to the sequence (\ref{nm1}); this observation  is 
a special case of  Lemma \ref{resper} below. Thus it is helpful to consider the three 
trisections $(\,s_{3k+i}(110)\,)$ for $i=0,1,2$, each of which satisfy the 
second-order recurrence 
\beq\label{tris} 
s_{3(k+2)+i}(110) - 1330670\,s_{3(k+1)+i}(110)+ s_{3k+i}(110)=0, 
\eeq 
as follows by applying the formula (\ref{diss}).
The easiest case is $i=1$, since $s_{3k+1}\equiv 0 \pmod{111}$ for all 
$k$; 
so in this subsequence, the first term $111=3\times 37$ is composite, and subsequent terms 
$ 147704481=111\times 1330671$, $196545921732159=111\times 1770683979569$, etc. are all 
multiples of 111. The trisection $(\,s_{3k}(110)\,)$ is the subsequence beginning with 
$s_0(110)=1$, and then $s_3(110)=1342879=9661\times 139$, 
$s_6(110)=1786928798929=116876761\times 15289$, and by 
induction it can be shown that each of these terms is divisible by the 
corresponding one for the sequence $(\,s_{3k}(5)\,)=1,139,15289,\ldots$, so 
that 
\beq\label{trifac} 
s_{3k}(110) = R_{3k}(5) \, s_{3k}(5), 
\eeq 
where the integer sequence of prefactors satisfies the third order recurrence 
\beq\label{3rdo} 
R_{3(k+3)}(5) -12099\, \Big( R_{3(k+2)}(5) - R_{3(k+1)}(5)\Big) -R_{3k}(5) = 0.
\eeq 
Similarly, for the remaining trisection, namely 
$(\,s_{3k+2}(110)\,)$, one has 
\beq\label{trifac2} 
s_{3k+2}(110) = R_{3k+2}(5)\, s_{3k+2}(5), 
\eeq 
where the prefactor sequence $(\, R_{3k+2}(5)\,)$ consists of integers and satisfies the same recurrence 
(\ref{3rdo}). In fact, it is not necessary to consider this trisection separately, since its properties follow 
immediately from 
extending  $(\,s_{3k}(110)\,)$ to $k<0$ and using the symmetry (\ref{sym}). 
These observations show that all the terms in (\ref{n110s}) are composite 
for $k>0$, confirming that $a_{110}=-1$ as claimed. Moreover, 
for all $k$ there is a factorization 
\beq\label{trifac3} 
s_{k}(110) = R_{k}(5) \,s_{k}(5), 
\eeq
where 
\beq\label{3rdo2} 
R_{k+3}(5) -24\, \Big( R_{k+2}(5) - R_{k+1}(5)\Big) -R_{k}(5) = 0, 
\eeq 
but the prefactors making up the full sequence $(\,R_k(5)\,)_{k\geq 0}$, that is 
$$1,\frac{37}{2},421,9661,\frac{443557}{2},5091241,116876761,
\frac{5366148517}{2}, \ldots,$$
are  only  integers in the cases (\ref{trifac}) and  (\ref{trifac2}), and not when $k\equiv 1 \pmod{3}$.
\end{example} 

The values of $n$ with $a_n=-1$ mentioned so far all have one thing in common: they 
correspond to values of dilated Chebyshev polynomials of the first kind. Indeed, the 
four -1 terms displayed in (\ref{anseq}) appear at the index values 
$$7=\Tc_2(3), \quad 14=\Tc_2(4), \quad 23=\Tc_2(5), \quad
34=\Tc_2(6), $$ 
and Theorem (\ref{t2}) implies that 
$s_k(n)$ is composite for all $k\geq 1$  when $n=\Tc_2(j)$, $j\geq 3$, while 
$$ 
110 = \Tc_3(5). 
$$ 
It turns out that for any Chebyshev value $n=\Tc_p(j)$ with $p>1$, there is a factorization analogous to (\ref{t2fac}) or (\ref{trifac3}): 
see Theorem \ref{factors} below. 
Due to the identity (\ref{tid}), it is sufficient to consider  the case of prime $p$ only. 

The curious reader might wonder why the values $n=18=\Tc_3(3)$ and 
$n=52=\Tc_3(4)$ are missing from the discussion. The reason is that, 
although there is a factorization 
analogous to  (\ref{trifac3}) for these values of $n$, there are the prime terms $s_1(18)=19$ and 
$s_1(52)=53$, which imply that $a_{18}=1=a_{52}$; 
but it turns out that there are no 
other primes in the sequences $(\, s_k(n)\, )_{k\geq 0}$ for $n=18$ or $52$. 
See Theorem \ref{maint} for a more general statement which includes all these 
Chebyshev values.

\section{General properties of the defining sequences} 

By writing the general  solution of (\ref{srec}) in terms of the roots of 
its characteristic quadratic, and  using various expressions for the 
dilated 
Chebyshev polynomials, as in 
section 2, 
we immediately obtain a number of equivalent explicit formulae for the 
sequence $(\,s_k(n)\,)$.
\begin{proposition} 
The terms of the sequence generated by (\ref{srec}) with the initial values (\ref{inits}) 
are given explicitly  by 
\beq\label{expl} 
s_k(n) = \frac{\la^{k+1} - \la^{-k}}{\la -{1}} 
= \Uc_{k-1}(n) + \Uc_k(n),
\eeq
where $\la$ is given in terms of $n$ according to (\ref{ladef}), and by 
\beq\label{mainf} 
s_k(n) = \Uc_{2k}(\sqrt{n+2}) 
= \frac{ \sin\Big( (2k+1)\theta /2 \Big) } 
{ \sin ( \theta /2 ) } , 
\eeq 
and they have the generating function 
\beq\label{gf} 
G(X,n): = \sum_{j=0}^\infty s_j(n)\, X^j =\frac{1+X}{1-nX+X^2}. 
\eeq
\end{proposition} 
\begin{proof}
The first formula in (\ref{expl}) is equivalent to (\ref{alform}), with $\la=\al^2$, and the other one  
follows by rewriting the Chebyshev polynomials as linear combinations of $\la^k$ and $\la^{-k}$, which 
generically provide two independent 
solutions of (\ref{srec}).\footnote{The first equality 
is invalid when $n=\pm 2$, due to  
repeated 
roots $\la=\la^{-1}=\pm 1$, 
cf.\ (\ref{lin}) and (\ref{nm2}). 
} 
For the latter set of identities, 
let $m=\sqrt{n+2}$, and note that $\Tc_2(m)=n$, so  $\theta$ can always be chosen such that 
$m=2\cos(\theta/2)$. 
The expression on the far right-hand side of (\ref{mainf}) is obtained by 
 by applying (\ref{sine}) to 
the last equality in (\ref{expl}), or by setting 
$\al=e^{\ri\theta/2}$ in (\ref{alform}), 
and this expression equals $\Uc_{2k}(2\cos(\theta/2))=\Uc_{2k}(m)$.
The generating function (\ref{gf}) follows from using the 
first formula in (\ref{expl}) and summing a pair of geometric series. 
\end{proof}

\begin{remark} \label{altf} 
The last formula in (\ref{expl}) together with (\ref{rsrel})  shows that
the terms on the right-hand side of the factorization (\ref{t2fac}) 
in the case $n=\Tc_2(j)=j^2-2$ can also be written as 
\beq\label{rsu} 
r_k(j) =  \Uc_k(j)- \Uc_{k-1}(j), \qquad 
s_k(j) = \Uc_k(j)+  \Uc_{k-1}(j). 
\eeq 
\end{remark}

If $5/2<n\in\R$ then $\la>2$, so $\la^{-k}/(\la -1)<1$ for all $k\geq 0$, and so we have 

\begin{corollary} \label{floorf}
For all real $n>5/2$, the terms $s_k(n)$ for $k\geq 0$ are given by 
$$ 
s_k(n)= \left \lfloor{\frac{\la^{k+1}}{\la-1}}\right \rfloor .
$$ 
\end{corollary}

The recurrence (\ref{srec}) can also be rewritten in matrix form, as 
\beq\label{matf}
{\bf v}_{j} = {\bf A}\, {\bf v}_{j-1} , 
\eeq 
where
$$
{\bf A} = 
\left(\begin{array}{cc} 0 & 1 \\ -1 & n \end{array}\right), \qquad
{\bf v}_{j} = \left(\begin{array}{c} s_{j}(n) \\ s_{j+1} (n) \end{array}\right) , 
$$ 
 hence for all $j$ the terms of the sequence are given in terms of the powers of $\bf A$ by  
$$ 
{\bf v}_{j} = {\bf A}^{j}\, {\bf v}_{0} .
$$ 
By a standard method of repeated squaring, this allows 
rapid calculation of the terms of the sequence. 
\begin{proposition}\label{mata} 
The $j$th power of the matrix $\bf A$ is given explicitly by
\beq\label{pow}
 {\bf A}^{j} = \left(\begin{array}{cc} -\Uc_{j-2}(n) & \Uc_{j-1}(n) \\ -\Uc_{j-1}(n)  & \Uc_{j}(n)  \end{array}\right),
\eeq 
and this can be calculated  in $O(\log j )$ steps. 
\end{proposition} 
\begin{proof} 
The formula (\ref{pow}) follows by induction, noting that the columns of the matrix on the right-hand side 
satisfy the same recurrence (\ref{matf}) as the vector ${\bf v}_j$, and it is trivially true for $j=0$. 
To calculate the powers of $\bf A$ quickly, compute the binary expansion $j=\sum_{i=0}^{d-1}b_i\,2^i$, 
where $b_{d-1}=1$ and $d=\log_2 j+1$ is the number of bits, then use repeated squaring to obtain the 
sequence $ \tilde{{\bf A}}_i={\bf A}^{2^i}$ for $i=0,1,\ldots, d-1$, and finally evaluate 
${\bf A}^j =\prod_{i=0}^{d-1} \tilde{{\bf A}}_i^{b_i}$.
\end{proof} 

There are other useful 
representations for the terms $s_k(n)$, two of which we record in the following 
\begin{proposition}\label{expa} 
For $k\geq 0$, the terms of the sequence $(\,s_k(n)\,)$ admit the expansion 
\beq\label{nexp} 
s_k(n) = 
\sum_{i=0}^{\left \lfloor{\frac{k}{2}}\right \rfloor}(-1)^i  { k-i \choose i}
n^{k-2i} + \sum_{i=0}^{\left \lfloor{\frac{k-1}{2}}\right \rfloor}(-1)^i  
 { k-i -1\choose i}
n^{k-2i-1}
\eeq 
in powers of $n$, and the expansion 
\beq\label{texp} 
s_k(n) = \frac{1}{2}\Tc_0(n) + \sum_{i=1}^k \Tc_i(n)
\eeq 
in terms of dilated Chebyshev polynomials of the first kind. 
\end{proposition} 
\begin{proof} The first expansion (\ref{nexp}) follows from the expression on the far right-hand side of (\ref{expl}), 
together with equation (\ref{uexp}).  The second expansion (\ref{texp}) corresponds to a standard identity 
for the Dirichlet kernel; it   can be proved 
by noting that dilated first/second kind Chebyshev 
polynomials are related via  
the identity 
$ 
\Tc_k(n) = 2\Uc_k(n) - n\Uc_{k-1}(n)$, which is easily verified.  
Taken together with the recurrence (\ref{urec}), 
as well as the last expression in  
(\ref{expl}), this gives  
\beq\label{diff} 
s_k(n) - s_{k-1}(n) = \Tc_k(n). 
\eeq 
Thus the expansion (\ref{texp}) is obtained by starting from $s_0=1=\frac{1}{2}\Tc_0$ and then taking the 
 telescopic sum of the first difference formula (\ref{diff}). 
\end{proof} 
\begin{remark} A different form  of series expansion for   the airfoil polynomials of the second kind is given in \cite{nasa}. 
\end{remark} 

\begin{proposition}\label{pshift}
For any odd integer $p$, 
\beq\label{psh} 
s_{k+p}(n)  - s_k(n) = \Tc_{k+\frac{p+1}{2}}(n) \, s_{(p-1)/2}(n). 
\eeq 
\end{proposition}  
\begin{proof} This follows from the trigonometric  expression on the far right-hand side of (\ref{mainf}), by applying the addition formula 
(\ref{sine}). 
\end{proof} 
\begin{remark} The formula (\ref{diff}) is the particular case $p=1$ of the above identity. 
\end{remark}

\begin{proposition}\label{linrels} 
Any  decimation 
$(\, s_{i+dk}(n)\,)$  of the  sequence  of order $d$, satisfies the linear recurrence 
\beq\label{lind} 
s_{i+d(k+1)}(n) - \Tc_d(n)\, s_{i+dk}(n) + s_{i+d(k-1)}(n) = 0. 
\eeq    
\end{proposition} 
\begin{proof} By the first formula for $s_k(n)$ in  (\ref{expl}), the terms of the decimated sequence can be written 
as linear combinations of $k$th powers of $\la^d$ and $\la^{-d}$, so from the formula (\ref{diss}) we find 
$$ 
\Big(\rS^2 - (\la^d+\la^{-d})\, \rS +1\Big) \,s_{i+dk}(n) =0, 
$$ 
and by using (\ref{ladef}) we see that $\la^d+\la^{-d} = 2\cos(d\theta) = \Tc_d(n)$, which verifies (\ref{lind}).  
\end{proof} 

It is worth highlighting some particular cases of the preceding two results, namely 
the formulae 
\beq\label{qpm1a} 
s_{2j}(n) +1= s_j(n)\, \Tc_j(n), 
\qquad s_{2j+1}(n) -1= s_j(n)\, \Tc_{j+1}(n),
\eeq
of which the first arises by setting $i=0$, $k= 1$, $d= j$ in (\ref{lind}), 
while the second comes from taking $k= 0$, $p=2j+1$ in (\ref{psh}).  For primality testing of a number $q$, it is often useful to have a factorization, or a 
partial factorization, of either $q-1$ or $q+1$ \cite{bls, pom}, and each of the 
identities in (\ref{qpm1a}) 
also has an analogue where the sign of the  $\pm 1$ term on the 
left-hand side is reversed. 

\begin{proposition} \label{qpm} For any integer $j$, 
\beq\label{qpm1b} 
s_{2j}(n) -1=(n+2)\,  r_j(n)\, \Uc_{j-1}(n), 
\qquad s_{2j+1}(n) + 1= (n+2)\, r_j(n)\, \Uc_{j}(n). 
\eeq
 \end{proposition} 
\begin{proof}
For the first identity in (\ref{qpm1b}), 
using $\la =\al^2 $ with $\al=e^{\ri\theta/2}$ and $n =\la+\la^{-1}$ yields 
$n+2=(\al+\al^{-1})^2$, and then from (\ref{ralform}) and the definition 
of the dilated Chebyshev polynomials of the second kind it follows 
that 
$(n+2) r_j(n)\, \Uc_{j-1}(n)$ is equal to 
$$ 
(\al+\al^{-1})^2\, \left(\frac{\al^{2j+1}+\al^{-(2j+1)}}{\al+\al^{-1}}\right)\, 
\left(\frac{\al^{2j}-\al^{-2j}}{\al^2-\al^{-2}}\right) 
=\left(\frac{\al^{4j+1}-\al^{-(4j+1)}}{\al-\al^{-1}}\right) - 1
$$ 
which is precisely $s_{2j}(n) -1$, by (\ref{alform}). The proof of the 
second identity is similar. 
\end{proof}

Another basic fact we shall use is that, with a suitable restriction on $n$,   $s_k(n)$ is 
monotone increasing with $k$. 

\begin{proposition} \label{inc} 
For each real $n\geq 2$, the sequence $(\,s_k(n)\,)$ is strictly increasing, and grows exponentially with 
leading order asymptotics 
$$ 
s_k(n)\sim   \frac{1}{2}\left( 1+ \sqrt{\frac{n+2}{n-2}}\right) \,\left(\frac{n+\sqrt{n^2 - 4}}{2} \right)^k 
\qquad 
as \quad k\to\infty , 
$$ 
for all $n>2$.  
\end{proposition}  
\begin{proof} For real $n\geq 2$, from (\ref{ladef}) we can set 
$$\tau=\ri\theta = \log\left(\frac{n+\sqrt{n^2 - 4}}{2} \right), $$
which defines a bijection   
from the interval $n\in [2,\infty )$ to  $\tau\in [0,\infty)$. 
The  inverse is 
$$ 
n = 2\,\mathrm{cosh} \tau \implies \frac{\rd n}{\rd\tau}=2\,\mathrm{sinh} \tau >0, 
$$ 
and we have 
$$ 
\Tc_k(n) = 2\, \mathrm{cosh} (k\tau ) \implies \frac{\rd }{\rd\tau}\,\Tc_k(n) =2k\,\mathrm{sinh} (k\tau) >0 
$$ 
for $\tau>0$; hence, for all fixed $k$, $\Tc_k(n)$ is a strictly increasing function of $n$ for $n\geq 2$. Similarly, 
$\frac{\rd }{\rd k}\,\Tc_k(n) =2\tau\,\mathrm{sinh} (k\tau)$ so for all fixed $n>2$, the sequence $(\,\Tc_k(n)\,)_{k\geq 0}$ is also strictly increasing  
with $k$.
Then since $\Tc_k(2) =2$ for all $k$, it follows that, for all $k$,  
\beq\label{tb} 
\Tc_k(n) \geq 2\qquad \forall n \geq 2, 
\eeq 
so  from (\ref{diff}) 
we have $$s_k(n) - s_{k-1}(n)\geq 2.$$ 
Upon taking the leading term of the   explicit expression in terms of $\la$ in (\ref{expl}) and rewriting it as a 
function of $n$, the 
asymptotic formula results.  
\end{proof} 

We can now use the explicit formulae above to derive various arithmetical properties of the integer sequences defined by 
$s_k(n)$ for positive integers $n$. 
This will culminate in Lemma \ref{gcd} 
below, which describes coprimality conditions on the terms, as well as Lemma 
\ref{pdiv} and its corollaries, which constrain where particular prime factors can appear. 
To begin with we describe where primes can appear in the sequence. 
\begin{lemma}\label{2kp1} 
For all integers $n\geq 2$, if $s_k(n)$ is prime then $k=(p-1)/2$ for $p$ an odd prime. 
\end{lemma} 
\begin{proof}  
If $2k+1=ab$ is composite, for some odd integers $a,b\geq 3$, then the identity (\ref{uid}) can be applied to the middle expression in 
(\ref{mainf}), to write
$s_k(n)$ as the product 
$$ 
\begin{array}{rcl} 
s_k(n) & = & 
 \Uc_{a-1}\left(\Tc_b(\sqrt{n+2})\right) \,\Uc_{b-1}(\sqrt{n+2}) \\ 
& = &
 s_{(a-1)/2}\left(\Tc_b (\sqrt{n+2})^2 -2\right) \, s_{(b-1)/2}(n). 
\end{array} 
$$ 
Then, since $\Tc_2(j)=j^2-2$, by using (\ref{tid}) we have 
\beq\label{facid} \begin{array}{rcl} 
 s_k(n)
& = & s_{(a-1)/2}\left(\Tc_{2b} (\sqrt{n+2})\right) \, s_{(b-1)/2}(n) \\ 
& = &  s_{(a-1)/2}\left(\Tc_{b} (n)\right) \, s_{(b-1)/2}(n) , 
\end{array} 
\eeq 
and each factor above is an  integer greater than 1. 
\end{proof} 

Henceforth we will consider only integer values of $n$. It is well known that all linear recurrence sequences defined over $\Z$ are 
eventually periodic $\bmod \,m$ for any modulus $m$ \cite{ward}; and for the recurrence (\ref{srec}) we can say further that it is strictly periodic  $\bmod \,m$, because the linear map $(s_k,s_{k+1})\mapsto (s_{k+1},s_{k+2})$ defined by the matrix ${\bf A}$ in (\ref{matf}) is always invertible  $\bmod \,m$ (since $\det{\bf A} =1$). 
However, in order to obtain  coprimality conditions, we need a  lemma that explicitly describes the periodicity of the terms $s_k(n)\,\bmod s_{j}(n)$ for fixed $j$.

\begin{lemma}\label{resper}
For all integers $n\geq 2$ and  any odd number $p\geq 3$, 
the sequence of residues $s_k(n)\,\bmod s_{(p-1)/2}(n)$  is periodic with period $p$, and 
$s_k(n) \equiv 0 \pmod{s_{(p-1)/2}(n)}$ if and only if $k \equiv (p-1)/2 \pmod{p}$. 
\end{lemma} 
\begin{proof} The identity (\ref{psh}) implies that, for all $k$, 
$$s_{k+p}(n)\equiv s_k(n) \pmod{s_{(p-1)/2}(n)},$$ 
so the residues repeat with period $p$. 
By the monotonicity result in Proposition \ref{inc}, 
$$ 
1=s_0(n) <s_1(n)< \cdots < s_{(p-3)/2}(n) < s_{(p-1)/2}(n). 
$$ 
Then the symmetry (\ref{sym}) implies that the residues $\bmod \,s_{(p-1)/2}(n)$ are non-zero in the range
$-(p-1)/2 \leq k\leq (p-3)/2$, so the rest of the statement follows from the periodicity. 
\end{proof}

\begin{lemma}\label{star} 
For each integer $n\geq 2$ and any odd integer $p\geq 3$,  $s_k(n)$ is coprime to $s_{(p-1)/2}(n)$ if and only if 
$(p-1)/2-k$ is coprime to $p$. 
\end{lemma} 

\begin{proof}
Once again, we drop the argument $n$ for the purposes of the proof, 
and perform  induction on 
the odd integers $p\geq 3$. With $p$  fixed, for each $k$ it will be convenient to consider 
\beq\label{mdef} 
m = (p-1)/2-k .
\eeq 
For the base case $p=3$,  note  that the sequence of $s_k\,\bmod \,s_1$ repeats with period 3, by  Lemma \ref{resper}, and clearly $\gcd({s_0,s_1})=1=\gcd({s_{-1},s_1})$ 
so the pattern is $s_k\equiv -1,1,0\pmod{s_1}$ for $k\equiv-1,0,1\pmod{3}$; 
hence $\gcd({s_k,s_1})=1$ if and only if the quantity 
$m=1-k\not\equiv 0 \pmod{3}$, which is the required result in this case. Now we will assume that the result is true for all odd $q$ with 
$3\leq q <p$, and proceed to show that it is true for $p$.

 Firstly, if for some $k$ the corresponding value of $m$, given by (\ref{mdef}), is not coprime to $p$, then there is some odd $q$ with $3\leq q\leq p$, $q|m$ and $q|p$. Therefore we have 
$$ 
(q-1)/2 - k = (q-p)/2 +m \equiv (q-p)/2 \equiv 0 \pmod{q}. 
$$ 
So by Lemma \ref{resper}, both $s_k\equiv 0 \pmod{s_{(q-1)/2}}$ and $s_{(p-1)/2}\equiv 0 \pmod{s_{(q-1)/2}}$, hence 
$s_k$ and $s_{(p-1)/2}$ are not coprime.  

Thus it remains to show that 
$$\gcd(m,p)=1\implies \gcd(s_{(p-1)/2-m},s_{(p-1)/2})=1.$$ 
Observe that, by Lemma  \ref{resper}, it is sufficient to verify this for values of $m$ between $1$ and $p-1$ (i.e.\ the non-zero residue classes   $\bmod \,p$). First consider $k=(p-1)/2-m$  lying in the range $0\leq k\leq (p-3)/2$: this 
can be written as $k=(q-1)/2$ for some odd positive integer $q$, and $\gcd(m,p)=1$ is equivalent to the requirement that $\gcd(q,p)=1$; 
so either $q=1$ and $\gcd(s_0,s_{(p-1)/2})=1$ is trivially true, 
or $3\leq q<p-2$ and  $\gcd(s_{(q-1)/2}, s_{(p-1)/2})=1$ holds by the inductive hypothesis. 
Now for the range $-(p-1)/2\leq k\leq -1$, the result follows by the symmetry $k \to -1-k$, using (\ref{sym}). Hence, by applying the shift $k\to k+p$ and using Lemma  \ref{resper}, the result is true for all integers $k$ such that 
$\gcd((p-1)/2-k,p)=1$. 
\end{proof}

In fact, it is possible to make a stronger statement about the common factors of the terms of the sequence.

\begin{lemma}\label{gcd} 
For all integers $n\geq 2$ and $j,k\geq 0$, 
$$ 
\gcd\Big(s_j(n),s_k(n)\Big)=s_m(n), \qquad \mathrm{where} \quad 2m+1=\gcd(2j+1,2k+1).
$$
\end{lemma} 

\begin{proof} 
Given any $j,k$, suppose that  $2m+1=\gcd(2j+1,2k+1)$. 
The case $m=0$ follows from Lemma \ref{star}, taking $p=2j+1$. If $m>0$, then 
by writing $2j+1=(2m+1)(2j'+1)$, $2k+1=(2m+1)(2k'+1)$ with $\gcd(2j'+1, 2k'+1)=1$,  and applying (\ref{facid}), we have 
$$
\gcd\Big(s_j(n),s_k(n)\Big)=s_m(n) \, \gcd\Big(s_{j'}(\Tc_{2m+1}(n)),s_{k'}(\Tc_{2m+1}(n))\Big) = s_m(n), 
$$ 
by applying Lemma \ref{star} once again.
\end{proof} 

\begin{remark} The preceding result is a special case of a result on the greatest common divisor of a pair of Lehmer numbers: see Lemma 3 in \cite{stewart}. 
\end{remark}

\begin{remark} Since the argument $n$ plays a passive role in most of the above, it is clear that, mutatis mutandis,  
Lemmas \ref{resper},  \ref{star} 
and \ref{gcd} 
also apply to the sequence of polynomials $(\,s_k(n)\,)$ 
in $\Z[n]$. Analogous divisibility properties for the Chebyshev polynomials of the first kind are described in \cite{rtw}. 
\end{remark} 

The preceding results allow the periodicity of the sequence modulo any prime to be described quite precisely. 
The notation $\legendre{\,\cdot\,}{\,\cdot\,}$ is used below to denote the Legendre 
symbol. 

\begin{lemma}\label{pdiv} Let $n\geq 2$ be fixed, and for any prime $q$ let $\pi (q)$ denote the period of the sequence $(s_k(n) \bmod \,q)$. 
 Then $\pi (2)=3$ if and only if $n$ is odd, in which case  $s_k(n)$ is even$\iff k\equiv 1\pmod{3}$, 
while $\pi(2)=1$ and all $s_k(n)$ are odd when $n$ is even. Moreover, for  
 $q$ an odd prime,   one of three possibilities can occur: (i) $\legendre{n^2-4}{q}=\pm1$ and $\pi(q)|q\mp 1$; 
(ii) $n\equiv 2\pmod{q}$ and $\pi(q)=q$ with $s_k(n)\equiv 0\pmod{q}\iff q|2k+1$;  
(iii) $n\equiv -2\pmod{q}$ and $\pi(q)=2$ with $s_k(n)\equiv (-1)^k\pmod{q}$. 
\end{lemma} 
\begin{proof} When  $n$ is even, then since $s_0(n)=1$ and $s_1(n)=n+1$ are both odd, it follows from (\ref{srec}) 
that $s_k(n)$ is odd for all $k$, so $\pi(2)=1$.   For $n$ odd, $s_1(n)$ is even, so by 
 Lemma \ref{star}, 
$s_k(n)$ is even if and only if  $k\equiv 1\pmod{3}$, and $\pi(2)=3$. 

Now let $q$ be an odd prime. For case (i) it is most convenient to consider the behaviour of  $(s_k \bmod \, q)$ in terms of the equivalent sequence 
 defined by the recurrence (\ref{srec}) in the finite field $\F_q$. In that case we have $n>2$, 
and  when  $\legendre{n^2-4}{q}=1$ it follows 
that $n^2-4$ is a quadratic residue mod $q$, so  the 
first formula in (\ref{expl}), which can be rewritten as
\beq\label{forml}
s_k(n)= \frac{\la^{-k}(\la^{2k+1}-1)}{\la -1}, 
\eeq  
remains valid  in terms of $\la\in\F_q$, with $\la\neq \pm 1$, and $\la^{q-1}=1$ in  $\F_q$ by Fermat's little theorem. The terms 
of the sequence repeat with period $\pi(q)=\mathrm{ord}(\la)>2$, the multiplicative order of $\la$ in the group 
 $\F_q^*$, and this divides $q-1$ by Lagrange's theorem. The case  $\legendre{n^2-4}{q}=-1$ is similar, but now  
 $n^2-4$ is a quadratic nonresidue mod $q$, so $\la$ is not defined in $\F_q$ and the formula (\ref{forml}) 
should be interpreted in the field extension $ \F_q[\sqrt{n^2-4}] \simeq\F_{q^2}$. The Frobenius automorphism 
$\la\to\la^q$ exchanges the roots of the quadratic $X^2-n\,X+1=(X-\la)(X-\la^{-1})$, hence $\la^q=\la^{-1}$. 
Thus $\la^{q+1}=1$, and now the sequence given by  (\ref{forml})  
 repeats with period $\pi(q)=\mathrm{ord}(\la)$, the order of $\la$ in 
 $\F_{q^2}^*$, which divides $q+1$.  In case (ii), the sequence $s_k(n)\bmod \,q$ is the same as the 
sequence (\ref{lin}) mod $q$, which first vanishes when $k=(q-1)/2$ and repeats with period $q$, and in case 
(iii) the sequence is equivalent to (\ref{nm2}), which is never zero mod $q$. 
\end{proof}

At this stage it is convenient to introduce the notion of  a primitive prime divisor (sometimes just referred to as a primitive divisor), which is a 
prime factor $q$ that divides $s_k(n)$ but does not divide any of the previous terms in the sequence 
\cite{estw}, and by convention does not divide the discriminant $n^2-4$ either \cite{bilu, schinzel}. 

\begin{definition} Let the product of the discriminant and the first $k$ terms be denoted by 
\beq\label{pik} 
\Pi_{k}(n)=(n^2-4)\, s_1(n)s_2(n) \cdots s_{k}(n). 
\eeq 
A {\it primitive prime  divisor} of $s_k(n)$ is a prime $q|s_k(n)$ such that 
$q\not | \,\Pi_{k-1}(n)$.  
\end{definition} 

Case (i) of Lemma \ref{pdiv} is the most interesting one. 
In that case it is clear from  (\ref{forml}) that a prime $q|s_k(n)$ for some $k$ whenever 
$\la^{2k+1}=1$ in $\F_{q^2}\supset\F_q$, and then  $\pi(q)=\mathrm{ord}(\la)=2k^*+1$ must be odd, where 
$k^*=(\pi(q)-1)/2>0$ is the smallest 
$k$ for which this happens; and if $\pi(q)$ is even then this cannot happen. 
If we include $q=2$, then 
we can rephrase the latter by saying that the prime factors $q$ appearing in the sequence 
$(\,s_k(n)\,)$ are precisely those $q$ which have an odd period $\pi(q)>1$, and this 
consequence of Lemma \ref{pdiv} can be restated  in terms of primitive prime divisors. 

\begin{corollary} \label{papp}
A prime  $q$  is a primitive divisor of $s_k(n)$ if and only if $k=(\pi(q)-1)/2$ 
where  $\pi(q)$ is odd. 
Moreover, if $q$ is odd and $\legendre{n^2-4}{q}=\pm1$ then 
$q=2a\pi(q)\pm 1$ for some positive integer $a$. 
\end{corollary} 
 
The latter statement just says that an odd primitive divisor of $s_k(n)$ has the form 
$q=2a(2k+1)\pm 1$ for some $a\geq 1$, so the minus sign with $a=1$ gives the lower bound $q\geq 4k+1$.  Hence 
 the primes that do not appear as factors in the sequence can also be characterized. 
\begin{corollary} \label{rp} 
If a prime $q<4k+1$ is not a factor of $\Pi_{k-1}(n)$,  
then it never 
appears as a factor of $s_j(n)$  for $j\geq k$, and $\pi(q)$ is even.
\end{corollary} 

So far we have concentrated on properties of $s_k(n)$ for fixed $n$ and allowed $k$ to vary. However, 
if one is interested in finding factors of $s_k(n)$ for large $n$, then it may also be worthwhile 
to consider other values of $n$, as the following result shows. 

\begin{proposition} 
Suppose that an integer $f|s_k(n)$ for some $k,n$. Then $f|s_k(m)$ 
whenever $m\equiv n \pmod{f}$. 
\end{proposition} 
\begin{proof}
If $m\equiv n \pmod{f}$ then $m^j\equiv n^j \pmod{f}$ for any exponent $j\geq 0$, and 
since $s_k(m)$ is a polynomial in $m$ with integer coefficients, it follows that 
 $s_k(m)\equiv s_k( n)\equiv 0 \pmod{f}$, as required. 
\end{proof} 

\section{Generic factorization for Chebyshev values} 
\label{chebv}

The sequence $(\,s_k(n)\,)$ has special properties when $n$ is given by a dilated Chebyshev polynomial of the 
first kind evaluated at an integer value of the argument.  
\begin{theorem} \label{factors} 
For all integers $p\geq 2$, when $n = \Tc_p(j)$ for some integer $j$ 
the terms of the sequence  $(\,s_k(n)\,)$
can be factorized as a product of 
rational numbers, that is 
\beq\label{fac} 
s_k (\Tc_p(j)) = R_k(j) \, s_k(j)  ,
\eeq 
where the prefactors $R_k(j)\in\Q$ 
are given by 
\beq\label{rform} 
R_k(j) = \frac{\Uc_{p-1}(\Tc_{2k+1}(\sqrt{j+2}))}{\Uc_{p-1}(\sqrt{j+2})} 
\eeq 
and satisfy a linear recurrence of order $p$. 
In particular, for $p=2$ the prefactor is 
$R_k(j)=r_k(j)\in \Z$, 
as  given in Theorem \ref{t2},  
while for all odd $p$ the prefactor can be written as 
\beq \label{poddf} 
R_k(j) = \frac{s_{(p-1)/2}(\Tc_{2k+1}(j))}{s_{(p-1)/2}(j)} \in\Q, 
\eeq 
and satisfies the recurrence 
\beq\label{podd} 
(\rS-1)\prod_{i=1}^{(p-1)/2}(\rS^2 -\Tc_{2i}(j)\,\rS +1) \, R_k(j) =0.
\eeq 
\end{theorem} 
\begin{proof} Upon introducing $\phi$ such that $\ell = \sqrt{j+2}=2\,\cos(\phi /2)$, 
the formula (\ref{mainf}) gives  
$$ 
R_k(j) = \frac{s_k (\Tc_p(j)) }{s_k(j)} = \frac{\sin\Big((2k+1)p\phi/2\Big)\sin(\phi/2)} 
{\sin(p\phi/2)\sin\Big((2k+1)\phi/2\Big)}, 
$$ 
and  the definition of the dilated Chebyshev polynomials 
of the second kind in (\ref{1st}) produces (\ref{rform}). 
For integer $j$, $R_k(j)$ is a ratio of  integers, so it is a rational number 
(positive for $j\geq 2$).
In the case $p=2$, $R_k(j)=r_k(j)$, which can be written in the form 
\beq\label{p2f} r_k(j) = \frac{\Tc_{2k+1}(\sqrt{j+2})}{\sqrt{j+2}}, \eeq 
which  is an integer, as follows from the fact that $\Uc_1(\ell) =\ell$, 
and this ratio is an even polynomial of  degree $2k$ in $\ell$ with integer coefficients, hence it is a polynomial of degree $k$  in $j$; and by (\ref{tb}) it is positive for real $j\geq 2$, and takes positive integer values for integers $j$ in this range.  
In the case that $p$ is odd, the expression (\ref{poddf}) is found by applying the formula (\ref{mainf})  to the numerator and denominator of (\ref{rform}). 

To see that $R_k(j)$ satisfies a linear recurrence of order $p$, note that, upon setting $\mu = \exp(\ri\phi )$ and applying the first formula in (\ref{expl}) with $\la=\mu^p$,  the factorization (\ref{fac}) can be seen as a consequence of the elementary algebraic identity 
\beq\label{algi}
\frac{\mu^{p(k+1)}-\mu^{-pk}}{\mu^p -1} =
\left( 
\frac{ \sum_{j=0}^{p-1}\mu^{(k+1)(p-1)-(2k+1)j}}{\mu^{p-1}+\mu^{p-2}+\cdots + 1}
\right) 
\, \left(\frac{\mu^{k+1}-\mu^{-k}}{\mu -1}\right) ,
\eeq 
where the first factor on the right-hand side above is just $R_k(j)$.
Thus the denominator of the expression for $R_k(j)$ in (\ref{algi}) is $\sum_{j=0}^{p-1}\mu^j$, 
which is  independent of $k$, while the numerator is 
a linear combination of $k$th powers of the characteristic 
roots $\mu^{(p-1)}, \mu^{(p-3)}, \ldots, \mu^{-(p-3)}, \mu^{-(p-1)}$, giving a total of $p$ distinct roots. When $p$ is 
even, the roots come in $p/2$ pairs, namely $\mu^{\pm(2i-1)}$ for $i=1,\ldots,p/2$, which gives the characteristic 
polynomial 
$$
F(\la) = \prod_{i=1}^{p/2}(\la^2-\Tc_{2i-1}(j)\la+1),
$$ 
so that, in particular, for $p=2$ the recurrence satisfied by $R_k(j)$ is (\ref{p2}), 
while for $p$ odd there are the pairs  $\mu^{\pm 2i}$ for $i=1,\ldots,(p-1)/2$ together with the root 1, 
which yields (\ref{podd}). 
\end{proof}

\begin{theorem}\label{maint}
Let $(a_n)_{n\geq 1}$ be the sequence  specified by Definition \ref{adef}. If 
$n = \Tc_2(j)$ for some $j\geq 3$, then $a_n = -1$. Furthermore, if 
$n = \Tc_p(j)$ for some $j\geq 3$ with $p$ an odd prime, then either 
$s_{(p-1)/2}(n)$ is not prime and $a_n = -1$, or $s_{(p-1)/2}(n)$ is 
the only prime in the sequence $(\,s_k(n)\,)_{k\geq 0}$ and  
$a_n = (p-1)/2$. 
\end{theorem}

\begin{proof} First of all, consider the factorization (\ref{fac}) when $p=2$, with prefactor $r_k(j)$ as in Theorem \ref{t2}, given by (\ref{p2f}). 
When $j=2$ this is not interesting, because it gives $r_k(j)=1$ for all $j$. However,  note the property (mentioned in passing in the proof of Proposition \ref{inc}), 
that for real $n>2$, the sequence $(\,\Tc_k(n)\,)_{k\geq 0}$ is strictly 
increasing with the index. Hence, for all $k>0$, $\Tc_{2k+1}(\sqrt{j+2})>\sqrt{j+2} = \Tc_{1}(\sqrt{j+2})$. Thus for all 
$j\geq 3$ and $k\geq 1$, 
both factors $r_k(j)$, $s_k(j)$ are greater than 1, so $s_k(T_2(j))$ can never be prime, and $a_n=-1$. 

Now for any odd prime $p$, note that, a priori, the prefactor $R_k(j)$  in  (\ref{fac})  is a positive rational number, 
and the formula (\ref{poddf}) gives 
\beq\label{sikid} 
s_k (\Tc_p(j)) = \frac{ s_k(j)  \, s_{(p-1)/2}(\Tc_{2k+1}(j))}{s_{(p-1)/2}(j)}. 
\eeq 
 However, according to Lemma \ref{resper}, $s_{(p-1)/2}(j)|s_k(j)$ whenever $k\equiv (p-1)/2 \pmod{p}$. 
On the other hand, for all other values of  $k\not\equiv (p-1)/2 \pmod{p}$, Lemma \ref{star} says that 
$\gcd(s_k(j) , s_{(p-1)/2}(j))=1$, therefore $s_{(p-1)/2}(j)$ divides $s_{(p-1)/2}(\Tc_{2k+1}(j))$ and $R_k(j)\in\Z$. 
Thus, for all $k$,  the terms $s_k (\Tc_p(j))$ can be written as a product of two integers, that is 
\beq\label{cases} 
s_k (\Tc_p(j)) = \begin{cases}
 \hat{R}_k(j)\, s_{(p-1)/2}(\Tc_{2k+1}(j)), & k\equiv (p-1)/2 \pmod{p}; \\
 R_k(j)\,s_k(j), & \text{otherwise}, 
\end{cases} 
\eeq 
where $R_k(j)$ is given by (\ref{poddf}) as above, and  
\beq\label{rt} 
\hat{R}_k(j) =  \frac{ s_k(j) }{s_{(p-1)/2}(j)} = s_i(\Tc_p(j)) \quad \mathrm{for}\,\, k=(p-1)/2+ip,  
\eeq 
with the latter formula being obtained from (\ref{facid}). 
In the first case of (\ref{cases}) above, for $k=(p-1)/2$ the prefactor is $\hat{R}_{(p-1)/2}(j)=1$, while 
$\hat{R}_{k}(j)>1$ for all $k=(p-1)/2+i p$, $i\geq 1$,  
by Lemma \ref{inc}, and the other factor is $s_{(p-1)/2}(\Tc_{2k+1}(j))>1$ for all these values of $k$. 
In the second case, for $k>0$, 
we can use (\ref{texp}) together with (\ref{tid}) to write 
$$
\begin{array}{rcl}  
s_{(p-1)/2}(\Tc_{2k+1}(j)) & = & \frac{1}{2}\Tc_0 + \sum_{i=1}^{(p-1)/2}\Tc_i(\Tc_{2k+1}(j)) \\
& = & \frac{1}{2}\Tc_0 + \sum_{i=1}^{(p-1)/2}\Tc_{(2k+1)i}(j) \\ 
& > & \frac{1}{2}\Tc_0 + \sum_{i=1}^{(p-1)/2}\Tc_{i}(j) = s_{(p-1)/2}(j) , 
\end{array} 
$$
so $s_{(p-1)/2}(\Tc_{2k+1}(j)) / s_{(p-1)/2}(j)>1$. 
Hence both factors $R_k(j)$, $s_k(j)$ are greater than 1 in the second case of (\ref{cases}). Thus 
the only term that can be prime  is  $s_{(p-1)/2} (\Tc_p(j))$, 
and the result is proved. 
\end{proof} 

\begin{remark} 
For any odd $p=2i+1$, the identity (\ref{sikid}) can 
be rewritten in the symmetric form 
\beq\label{symid} 
s_{i}(j)\,
s_k (\Tc_{2i+1}(j))  =  s_k(j)  \, s_{i}(\Tc_{2k+1}(j))  . 
\eeq  
\end{remark} 

\begin{remark} \label{polj} 
Similarly to the remark after Lemma \ref{gcd}, the formula (\ref{cases}) also corresponds to factorizations of 
the corresponding polynomials in $\Z[j]$, according to whether $k\equiv (p-1)/2 \pmod{p}$ or not.
\end{remark} 

It is clear from the factorizations (\ref{cases}) that in the second case, $s_k(\Tc_p(j)) \equiv 0 \pmod{s_k(j)}$ whenever $k$ is not 
congruent to $(p-1)/2 \,\bmod \,p$. It turns out that an explicit expression for   
$s_k(\Tc_p(j)) \,\bmod \,s_k(j)$ can be given  in the first case as well. Before doing so, it is convenient to define 
some more polynomials, which are shifted versions of the airfoil polynomials.

\begin{definition} \label{ppol} 
Polynomials $\Pc_k(z)$  are defined as elements of $\Z[z]$  by 
$$
\Pc_k(z) = s_k(2-z), 
$$ 
or equivalently by  
\beq\label{sinf}  
\Pc_k(4\sin^2\theta) = \frac{\sin\Big((2k+1)\theta\Big)}{\sin\theta}. 
\eeq 
They satisfy the linear recurrence 
\beq\label{prec}
\Pc_{k+1}(z)  +(z-2) \Pc_k(z) +\Pc_{k-1}(z) =0, 
\eeq 
and 
for $k\geq 0$ their expansion in powers of $z$ takes the form 
\beq\label{expan}
\Pc_k(z) = 2k+1 - c_k^{(1)}\, z +c_k^{(2)}\, z^2+\cdots+ (2k+1)(-z)^{k-1} + (-z)^k. 
\eeq 
with 
$$ 
c_k^{(1)}=\frac{k(k+1)(2k+1)}{3!} ,\qquad c_k^{(2)}= \frac{k(k-1)(k+1)(2k^2+5k+2)}{5!}.
$$ 
\end{definition} 

\begin{theorem} \label{rema} 
For all odd integers $p$, 
\beq\label{pform} 
s_k(\Tc_p(j)) = s_i(\Tc_p(j))\,\Pc_{(p-1)/2}\Big((2-j)\, s_k(j)^2\Big) \quad \mathrm{for} \quad k=(p-1)/2+ip,  
\eeq 
and in particular, 
$$ 
 s_k(\Tc_p(j))  \equiv  p\,  s_i(\Tc_p(j)) \pmod{(j-2)\,s_k(j)^2}
$$
holds in that case. 
\end{theorem}
\begin{proof}Upon setting $p=2q+1$, 
by  using (\ref{rt}) together with  the first formula in (\ref{cases}), 
we 
have 
$$ 
\begin{array}{rcl}
{s_k(\Tc_p(j))}/{s_i(\Tc_p(j))} & = & 
{\sin\Big((2q+1)(2k+1)\phi/2\Big)}\bigg/{\sin\Big((2k+1)\phi/2\Big)} \\
& = & \Pc_q(z), \qquad 
\mathrm{where} \qquad z= 4\sin^2\Big((2k+1)\phi/2\Big), 
\end{array} 
$$ 
and   (with the same notation as in the proof of Theorem \ref{factors})  we also have 
$j=2\,\cos\phi$. Comparing the expressions for $z$ and $j$ gives 
$
z =4\sin^2(\phi/2) \, s_k(j)^2 = (2-j)\, s_k(j)^2$, 
which yields the identity (\ref{pform}) in terms of the shifted airfoil polynomial $\Pc_{(p-1)/2}$. The 
 terms displayed in the 
expansion (\ref{expan}) are easily obtained from the recurrence (\ref{prec}), 
or by substituting $n=2-z$ in (\ref{nexp}), 
and  the  leading term 
gives the reduction  of (\ref{pform}) $\bmod \, s_k(j)^2$. 
\end{proof} 

We now turn to the sequences $(\,r_k(n)\,)$ for $n>0$, which are associated with negative values of $n$ via (\ref{rsrel}). 
It turns out that these sequences also admit factorizations for certain  Chebyshev values of $n$. The case of 
$n=\Tc_p(j)$ for odd index $p$ can be inferred immediately from Theorem \ref{factors} together with 
(\ref{rsrel}), since $\Tc_p$ is an odd function 
of its argument  in that case. However, the even case $n=\Tc_2(j)$ does not translate directly to 
the sequences $(\,r_k(n)\,)$, and 
requires a separate treatment. That there should be a significant difference for values of even Chebyshev polynomials 
is also 
apparent from comparison of  the fact that 
$a_7=a_{14}=a_{23}=a_{34}=-1$ in 
(\ref{anseq}), 
but  $\tilde{a}_1=\tilde{a}_2=\tilde{a}_{34}=-1$ in (\ref{atnseq}), while 
$\tilde{a}_{7}$, $\tilde{a}_{14}$ and $\tilde{a}_{23}$ are all positive. 

The following analogue of Theorem \ref{t2} for the sequences  $(\,r_k(n)\,)$ only provides 
a factorization of the terms for a particular subset of the values $n=\Tc_2(j)$ . 

\begin{theorem}\label{rt2} When $n=\Tc_2(j)=j^2-2$ with $j=2(\ell^2-1)$ for integer $\ell\geq 2$, 
the terms of the sequence $(\,r_k(n)\,)$ admit the factorization 
\beq\label{rt2fac} 
r_k(j^2-2) = f_k^+(j)\, f_k^-(j), 
\eeq 
where 
\beq\label{fform} 
f_k^{\pm}(j) = \frac{\ell r_k(j) \pm \delta_k}{\ell\pm 1}\in\Z, 
\qquad \delta_k = (-1)^{\left \lfloor{\frac{k+1}{2}}\right \rfloor}. 
\eeq 
\end{theorem} 
\begin{proof} Since $\delta_k^2=1$ and
 $r_k(j) =(j+2)^{-1/2} \left(\al^{(2k+1)/2}+ \al^{-(2k+1)/2}\right)$, using 
$\al^{1/2}+\al^{-1/2}=\sqrt{j+2}$ as in (\ref{dos}), 
it follows that 
$$ 
\begin{array}{rcl}
 f_k^+(j)\, f_k^-(j) &= & 
{(\ell^2 r_k(j)^2 -1)}/{(\ell^2-1)} =  j^{-1}\Big((j+2)r_k(j)^2-2\Big) \\ 
& = & 
{\Big( \left(\al^{(2k+1)/2}+ \al^{-(2k+1)/2}\right)^2-2\Big)}/{(\al+\al^{-1})}, 
\end{array} 
$$ 
which coincides with the formula (\ref{ralform}) for $r_k(n)$ with $n=j^2-2=4\ell^4-8\ell^2+2$ in this case. 
To see that each factor $f_k^\pm (j)$ is an integer for all $k$, note that 
$(\rS^2-j\rS+1)r_k(j)=0$,  and checking the sequence of signs $+1,-1,-1,+1$ for 
$k=0,1,2,3$ shows that  $(\rS^2-j\rS+1)\delta_k =j \delta_{k-1}$. Hence 
the factors in (\ref{t2fac}) each satisfy an inhomogeneous linear 
recurrence of second order, that is 
\beq\label{finh}
f_{k+2}^\pm(j) - j \, f_{k+1}^\pm(j) +f_k^\pm(j) = 2(\pm\ell - 1)\delta_{k-1}. 
\eeq 
From (\ref{fform}) it can be seen that 
$$
f_{-1}^{\pm}(j) = f_{0}^{\pm}(j) =1
$$ 
provide integer initial values for (\ref{finh}) in each case, 
so these two sequences consist entirely of integers. 
\end{proof} 

\begin{example} 
For $\ell=2$, the above result gives $j=\Tc_2(2\sqrt{2})=6$ and $n=\Tc_2(6)=34$ ,
with the sequence $(\,r_k(34)\,)=(\,f_k^+(6)\,f_k^-(6)\,)$ beginning 
\beq\label{r34} 
1,33,1121,38081,1293633,\ldots, 
\eeq 
where the factors are 
\beq\label{fpm6} 
(\,f_k^+(6)\,): \,\, 1,3,19,113,657,\ldots, \quad 
(\,f_k^-(6)\,): \,\, 1,11,59,337,1969,\ldots,
\eeq 
and these satisfy the inhomogeneous recurrences 
$$ \begin{array}{rcl} 
f_{k+2}^+(6) - 6 \, f_{k+1}^+(6) +f_k^+(6) 
& = & 2(-1)^{\left \lfloor{\frac{k}{2}}\right \rfloor} ,  \\
f_{k+2}^-(6) - 6 \, f_{k+1}^-(6) +f_k^-(6) 
& = & 6(-1)^{\left \lfloor{\frac{k}{2}}\right \rfloor+1} 
.
\end{array}
$$ 
\end{example} 

For the case where $n=\Tc_p(j)$ for $p$ odd, the formula (\ref{poddf}) can be applied, together with 
(\ref{rsrel}), to yield the factorization 
\beq\label{rfactor} 
r_k (\Tc_p(j)) = \tilde{R}_k(j) \, r_k(j), 
\eeq 
where 
\beq\label{Rform} 
 \tilde{R}_k(j) =R_{k}(-j) = \frac{r_{(p-1)/2}(\Tc_{2k+1}(j))}{r_{(p-1)/2}(j)} 
\eeq 
satisfies the same recurrence (\ref{podd}) as $R_k(j)$. 

\begin{example}
When $n=\Tc_3(3)=18$, the sequence $(\,r_k(18)\,)$ begins 
\beq\label{r18} 
1,17,305,5473,98209,1762289,31622993,567451585,\ldots, 
\eeq 
so that $\tilde{a}_{18}=1$ since $r_1(18)=17$ is prime. 
By adapting Theorem \ref{factors}, 
the terms can be factored as 
$$ 
r_k(18)=\tilde{R}_k(3)\,r_k(3), 
$$ 
where the sequence $(\,r_k(3)\,)$ begins with 
\beq\label{r3}
1,2,5,13,34,89,233,610,\ldots, 
\eeq 
and $(\,\tilde{R}_k(3)\,)$ is a sequence 
of rational numbers, starting with 
$$ 
1,\frac{17}{2},61,421, \frac{5777}{2}, 19801, 135721, \frac{1860497}{2}, \ldots, 
$$
which satisfies the third order recurrence 
$$ 
\tilde{R}_{k+3}(3)- 8 \, \Big(\tilde{R}_{k+2}(3) - \tilde{R}_{k+1}(3)\Big) - \tilde{R}_{k}(3)=0. 
$$ 
The prefactors $\tilde{R}_k(3)$ are integers whenever $k\equiv 0$ or $2\pmod{3}$, 
so $r_k(3)$ divides $r_k(18)$  for all such $k$, while 
the terms of the trisection $(\, r_{3k+1}(18)\,)$ are all divisible by  $r_1(18)$, which is the only 
prime in the sequence $(\,r_k(18)\,)_{k\geq 0}$.
\end{example}

Having described the situation for even Chebyshev values, and given 
the above example of an odd Chebyshev value, the analogue of 
Theorem \ref{maint} for $(\,r_k(n)\,)$ can  now be stated. 

\begin{theorem}\label{tmaint}
Let $(\tilde{a}_n)_{n\geq 1}$ be the sequence  specified by Definition \ref{atdef}. If 
$n = \Tc_2(j)$ where $j=2(\ell^2-1)$ for some $\ell\geq 2$, then $\tilde{a}_n = -1$. Furthermore, if 
$n = \Tc_p(j)$ for some $j\geq 3$ with $p$ an odd prime, then either 
$r_{(p-1)/2}(n)$ is not prime and $\tilde{a}_n = -1$, or $r_{(p-1)/2}(n)$ is 
the only prime in the sequence $(\,r_k(n)\,)_{k\geq 0}$ and  
$\tilde{a}_n = (p-1)/2$. 
\end{theorem}
\begin{proof}
For the case  $n = \Tc_2(j)$, $j=2(\ell^2-1)$ with $\ell\geq 2$, note that each factor in 
(\ref{rt2fac}) is an integer, and $r_0(\Tc_2(j))=f_0^\pm (j)=1$. Now as noted in the proof of Lemma \ref{inc}, 
for fixed argument the sequence of Chebyshev polynomials of the first kind is strictly increasing  
with the index $k\geq 0$, so from (\ref{p2f}) it follows that $(\,r_k(j)\,)_{k\geq 0}$ is strictly increasing. 
Thus, from their explicit expressions in (\ref{fform}),  both sequences $(\,f_k^\pm (j)\,)$ are 
 strictly increasing as well. Hence, for these values of $j$, $r_k(\Tc_2(j))$ is composite for $k\geq 1$. 
Hence there are no primes in the sequence $(\,r_k(n)\,)_{k\geq 0}$ for any of these even Chebyshev values of $n$. 

When $n = \Tc_p(j)$, $j\geq 2$ with $p$ an odd prime, there is the factorization (\ref{rfactor}), with 
$\tilde{R}_k(j) \in\Q$ given by (\ref{Rform}). Just as for the sequences $(\,s_k(n)\,)$, it is necessary 
to consider whether $k\equiv (p-1)/2 \pmod{p}$ or not. One can show 
that the analogue of Lemma \ref{resper} holds for the sequences $(\,r_k(n)\,)$, $n\geq 3$, 
so that when $k\equiv (p-1)/2 \pmod{p}$ the denominator of $\tilde{R}_k(j)$ divides $s_k(j)$; or, 
by using (\ref{rsrel}) together with (\ref{facid}), one can write the explicit factorization 
\beq\label{rifac}
r_k(\Tc_p(j))= r_i(\Tc_p(j))\, r_{(p-1)/2}(\Tc_{2k+1}(j)) \qquad \mathrm{for} \quad k=(p-1)/2+ip , 
\eeq 
where both integer factors above are greater than 1 for $i>0$. 
On the other hand, for the case $k\not\equiv (p-1)/2 \pmod{p}$, one can show that  
the analogue of Lemma \ref{star} also holds for the sequences $(\,r_k(n)\,)$, so 
in that case the denominator in (\ref{Rform}) must divide the numerator; hence 
$\tilde{R}_k(j) \in\Z$ and both factors in  (\ref{rfactor}) are integers. Then, similarly to 
the proof of Lemma \ref{inc}, setting $\sqrt{j+2}=2\,\mathrm{cosh}\tau$ for real $j\geq 2$  in (\ref{p2f}) gives 
$$r_k(j)= \frac{\mathrm{cosh}((2k+1)\tau)}{\mathrm{cosh}\tau}
\implies \frac{\rd}{\rd\tau}\, r_k(j)=\frac{2k\,\mathrm{sinh}((2k+1)\tau)}{\mathrm{cosh}\tau} + 
\frac{\mathrm{sinh}(2k\tau)}{\mathrm{cosh}^2\tau}>0
$$ 
for $k>0$, $\tau>0$, implying that $r_k(j)$ is a strictly increasing function of $j$ for $j\geq 2$.  
Hence $\tilde{R}_k(j)>1$ for $j>2$, so when  $0<k\not\equiv (p-1)/2 \pmod{p}$, both integer factors in  (\ref{rfactor}) are 
greater than 1. 
Thus $r_{(p-1)/2}(\Tc_p(j))$ is the only term that can be prime. 
\end{proof} 

\section{Appearance of primes for non-Chebyshev values}  

Theorem \ref{maint} says that for the values $n=\Tc_p(j)$ with prime $p$ and integer $j\geq 3$, the sequence $(\,s_k(n)\,)_{k\geq 0}$ contains at most one prime term, and this can only occur if $p$ is an odd prime, in which case $s_{(p-1)/2}(\Tc_p(j))$ is the only term that may be prime. It seems likely that these cases are exceptional, and for 
non-Chebyshev values of $n$ one would expect 
infinitely many prime terms, in line with general heuristic arguments for linear recurrence sequences \cite{epsw}.

\begin{figure} \centering
\includegraphics[width=10cm,height=10cm,keepaspectratio]{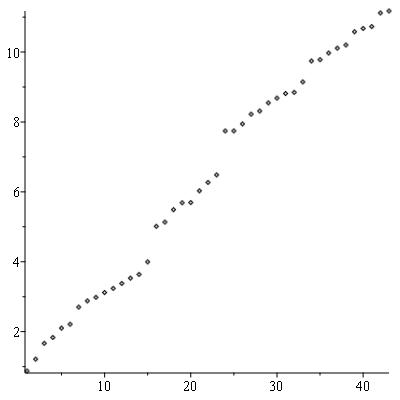}
\caption{\small{Plot of $\log\log s_{k_N}(n)$ against $N$ for the first 43 primes in the sequence for $n=3$.}}
\label{nis3}
\end{figure} 

\begin{conjecture} \label{sprimes}
Let $n>1$ be a positive integer. The sequence $(\,s_k(n)\,)_{k\geq 0}$ contains infinitely many primes if and only if $n\neq \Tc_p(j)$ for some prime $p$ and some integer $j\geq 3$. 
\end{conjecture}

In order to consider the distribution of primes in the sequence  $(\,s_k(n)\,)$, it is helpful to introduce some notation. 
Define a subsequence $(k_N)_{N\geq 1}$ of the positive integers by requiring that 
$$ 
s_{k_N}(n) = N\mathrm{th}\,\,\mathrm{prime} \,\,\mathrm{term} \,\, \mathrm{in} \,\,(\,s_k(n)\,)_{k\geq 0}, 
$$ 
and, for fixed $n$,  let 
$$
{\cal S}_k(n) =
 \{    \,\mathrm{prime}\, q \mid q<4k+1\,
\}  \cup
 \{    \,\mathrm{prime}\, q  \mid q| \Pi_{k-1}(n)\,
\}  
 , $$
where $\Pi_{k-1}(n)$ is the product in (\ref{pik}). 

\begin{conjecture} 
If $n\geq 3$ and 
 $n\neq \Tc_p(j)$ for some prime $p$ and some integer $j\geq 3$, 
then, as $N\to\infty$, 
\beq\label{conjf}
\log\log s_{k_N}(n) \sim C\, N, \qquad 
with \quad  C= e^{-\gamma} \, \log \sqrt{\la} , 
\eeq 
where $\la = \frac{n+\sqrt{n^2-4}}{2}$, and  $\gamma$ is 
 the Euler--Mascheroni constant. 
\end{conjecture}

The above assertion  is analogous to a conjecture of Wagstaff regarding Mersenne primes \cite{wagstaff}. If $M_N$ is the  $N$th prime of the form $2^k-1$, then the heuristic arguments of Wagstaff suggest that 
$$ 
\log\log M_N \sim C' \, N, \qquad \mathrm{with} \quad C'=e^{-\gamma}\log 2. 
$$ 
A very clear exposition of the statistical properties of Mersenne primes, with many plots,  
can be found on Caldwell's website \cite{caldwell}.\footnote{More specifically, see the page 
\url{https://primes.utm.edu/notes/faq/NextMersenne.html}}
When $n$ is a non-Chebyshev value, a heuristic derivation of the corresponding asymptotics of primes 
in the sequences $(\,s_k(n)\,)$ can be obtained in a similar way, as we now describe.

\begin{figure} \centering
\includegraphics[width=10cm,height=10cm,keepaspectratio]{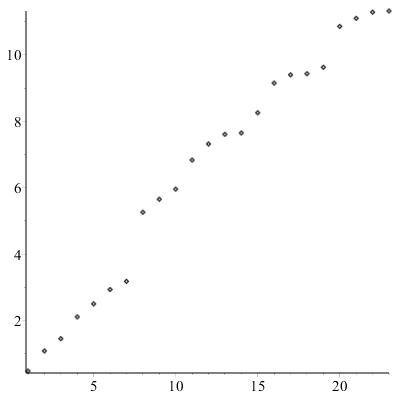}
\caption{\small{Plot of $\log\log s_{k_N}(n)$ against $N$ for the first 23 primes in the sequence for $n=4$.}}
\label{nis4}
\end{figure}

By Lemma \ref{2kp1}, if $s_k(n)$ is prime then $2k+1$ is prime, and by Lemma  \ref{star}, 
$s_k(n)$ is then coprime to $s_j(n)$ for all $1\leq j\leq k-1$, and for large enough 
$k$ it is also coprime to the discriminant $n^2-4$, hence it is coprime to $\Pi_{k-1}(n)$. 
On the other hand, by Corollary \ref{rp}, if $p=2k+1$ is prime then $s_k(n)$ is coprime to 
all primes $q<4k+1$: such primes are either primitive divisors of lower terms $s_j(n)$ with 
$j<k$, or they do not appear as divisors of the sequence at all. Thus no prime $q\in {\cal S}_k(n)$ can 
be a factor of $s_k(n)$ when   $2k+1$ is prime.
Then from the prime number theorem, for $k$ large, 
$$ 
\mathrm{ Prob}(2k+1 \,\mathrm{prime})\sim 2/\log(2k+1);$$ and, given that $2k+1$ is prime, the probability that $s_k(n)$ is prime is estimated by 
dividing by the probability that $s_k(n)$ is indivisible by primes $q$ that either divide lower terms in the sequence, 
or are forbidden from being divisors of $s_k(n)$ due to Corollary \ref{rp}, 
to yield 
$$ 
\begin{array}{rcl}
\mathrm{ Prob}(s_k(n) \,\mathrm{prime}|2k+1\,\mathrm{prime} ) & \sim & \frac{1}{\log s_k(n)} \,  
 \underset{q\in{\cal S}_{k}(n)}{\prod}\left(1-\frac{1}{q}\right)^{-1} \\
& \sim &   \frac{1}{k\log \la} \,  
 \underset{q\in{\cal S}_{k}(n)}{\prod}\left(1-\frac{1}{q}\right)^{-1},
\end{array}  
$$ 
where the latter expression comes from the asymptotics in Proposition \ref{inc}.  
By multiplying these two probabilities together, 
and using the limit 
$$  
\lim_{k\to\infty} \log k  \underset{q\leq k}{ \underset{q\,\mathrm{prime}}{\prod}}\left(1-\frac{1}{q}\right)=e^{-\gam}, 
$$ 
which is one of Mertens' theorems (see section 22.8 in \cite{hw}), 
gives 
$$
\lim_{k\to\infty} \log (2k+1)
  \underset{q< 4k+1}{ \underset{ q\, \mathrm{prime}}{\prod}}\left(1-\frac{1}{q}\right)\, \underset{q\geq 4k+1}{ \underset{\mathrm{prime}\, q\in{\cal S}_{k}(n)}{\prod}}\left(1-\frac{1}{q}\right)=e^{-\gam}
$$
which  produces the estimate 
$$ \left| \{\mathrm{prime}\,\,\mathrm{terms} \,\, s_k(n) \,\,\mathrm{for}\,\, 0<k\leq x\}\right| 
 \sim\frac{1}{ e^{-\gam}\, \log \sqrt{\la} } \sum_{k\leq x  }\frac{1}{k}
\sim C^{-1} \log x, $$
so if $s_{k_N}(n)$ is the $N$th prime term in the sequence then the formula (\ref{conjf}) follows from taking 
$$x=k_N\sim \frac{\log   s_{k_N}(n) }{\log\la}.$$

 Numerical evidence for small values of $n$ suggests that 
the log log plot of the prime terms in the sequence $(\,s_k(n)\,)$ is approximately linear (see e.g.\ Figure \ref{nis3} for the case $n=3$), and gives some
support for the proposed value of $C$. 
Moreover, it is expected that the appearance of prime terms should behave like a Poisson process, 
in complete analogy with Wagstaff's 
observations on the sequence of Mersenne primes \cite{wagstaff}. 
The first appendix below contains a list of the indices $k$ for the first probable primes that appear in the sequences  $(\,s_k(n)\,)$
for $n=3,4,5,6$, and as well as including the log log plots, in each of these cases a linear best fit value of $C$ is 
found, with the ratio 
$$ 
\rho(n)=\frac{C}{\log\sqrt{\la}}
$$ 
being compared with the value 
$$e^{-\gamma}\approx 0.561459$$ 
coming from Mertens' theorem. 

An analogous behaviour should be observed in the sequences $(\,r_k(n)\,)$ for positive $n$. 

\begin{conjecture} \label{rprimes}
Let $n>2$ be a positive integer. The sequence $(\,r_k(n)\,)_{k\geq 0}$ contains infinitely many primes if and only if 
$n\neq \Tc_p(j)$  for some prime $p$, where the  integer $j\geq 3$ takes one of values specified 
 in Theorem \ref{tmaint}. 
\end{conjecture} 

The first few prime terms in the sequence  $(\,r_k(3)\,)_{k\geq 0}$ are plotted in Figure \ref{nism3}; for more details see the first appendix. 

\section{Conclusions} 

It seems highly likely that Theorem \ref{maint} identifies all those values of $n\geq 3$ such that 
the sequence $(\,s_k(n)\,)_{k\geq 0}$ contains at most one prime, and Theorem \ref{tmaint} does 
the same for  $(\,r_k(n)\,)_{k\geq 0}$. 
The sequences corresponding to all other values of $n$ 
should have infinitely many prime terms, but proving this 
should be at least as difficult as proving that there are infinitely many Mersenne primes. 
For Lehmer numbers, the most sophisticated results currently available concern primitive divisors 
\cite{bilu, stewart2013}. 

The statistics of prime appearances for non-Chebyshev values of $n$ suggests a close analogy with 
Mersenne primes.
For Mersenne primes, the Lucas-Lehmer test is extremely efficient \cite{bruce}. The ideas 
from \cite{rod, rotk} can be adapted to yield a necessary condition  for  primality of $q=s_k(n)$, 
which can be tested efficiently, 
but to provide  sufficient conditions requires the use of a Lucas test or one of its generalizations \cite{bls, pom}, 
for which the formulae 
(\ref{qpm1a}) and  (\ref{qpm1b}) are useful, since 
 they provide partial factorizations of $q\pm 1$. 
 In  future we would like to consider some of the  large primes that appear in these sequences, 
extending the approach that was applied to the case $n=6$ in \cite{nsw}.

\section {Acknowledgements} 
ANWH is supported by EPSRC fellowship EP/M004333/1. 
Some results in sections 3,4 and 5 of this paper were also obtained independently 
by Bradley Klee, who provided useful suggestions for an early draft, and has developed 
a graphical calculator application to verify the factorizations in Theorem \ref{factors} for particular values of $p$ \cite{kleeapp}.  
We are grateful to David Harvey, Robert Israel, 
Don Reble, John Roberts, Igor Shparlinski and Neil Sloane for helpful comments. 
We are also extremely indebted to Hans Havermann, whose  
extensive numerical computations originally 
inspired many of the results described here. 

\section*{Appendix A: Sequences of prime appearances} 

In order to study the appearance of prime terms when 
$n$ is a non-Chebyshev value, for some particular small values of $n$ we calculated the possible 
prime terms $q=s_{(p-1)/2}(n)$ when $p=3,5,7,11,\ldots$ is an odd prime, and then tested them for primality using the Maple {\tt isprime} command. 
This uses a probabilistic test, which excludes certain composite values of $q$, 
while remaining $q$ are only pseudoprimes. For all but the largest  values of the index $k=(p-1)/2$, we also checked the computations 
with Mathematica's {\tt PrimeQ}$[q]$ command,  as well as performing a Lucas-Lehmer style test  for 
pseudoprimes of our own, and verified that the answer was the same,

For $n=3$, the list of the first 43  values $k$ for which $s_k(3)$ appear to be prime is  OEIS 
sequence \seqnum{A117522}, beginning 
$$ 
\begin{array}{l}
2, 3, 5,6,8,9,15,18,20,23,26,30,35,39,56,156,176,251,306,308,431,548, \\ 
680, 
2393, 2396, 2925,  3870,                              4233, 5345,                              6125,                              6981,                            7224, 9734, 
17724,                              18389,              \\              22253,                           25584,                             28001,                             40835,           44924,  47411, 70028, 74045.
\end{array}
$$ 
The (probable) primes $s_k(3)$ corresponding to these values of $k$ are listed in sequence \seqnum{A285992}. 
The log log plot of these terms is given in Figure \ref{nis3}. The slope of the best fit line for these points is  
$$ 
C= 0.2553739565.
$$

For $n=4$, the list of the first 23  values $k$ for which $s_k(4)$ appear to be prime is  
$$ 
\begin{array}{l}
1,2,3,6,9,14,18,146,216,293,704,1143,1530,1593,2924,7163,9176,9489, \\ 11531,39543,  50423,60720,62868, 
\end{array}
$$ 
which are listed in OEIS sequence \seqnum{A299100}, while the corresponding values $s_k(4)$ are given in \seqnum{A299107}. 
The log log plot of these terms is given in Figure \ref{nis4}. The best fit line for this set of points has slope 
$$ 
C= 0.5196737962. 
$$

For $n=5$, the list of the first 24  values $k$ for which $s_k(5)$ appear to be prime is  
$$ 
\begin{array}{l}
2,3,5,6,8,9,15,18,23,53,114,194,564,575,585,2594,3143,4578,4970, \\9261,11508,13298, 30018,54993, 
\end{array}
$$ 
as listed in  OEIS sequence \seqnum{A299101}, 
with the corresponding values of $s_k(5)$ listed as sequence \seqnum{A299109}.
The log log plot of these terms is given in Figure \ref{nis5}. The best fit line for this set of points has slope 
$$ 
C= 0.4568584420. 
$$

\begin{figure} \centering
\includegraphics[width=10cm,height=10cm,keepaspectratio]{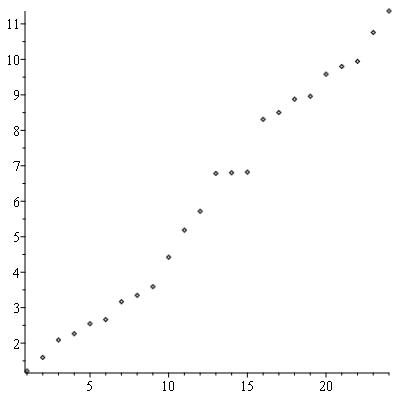}
\caption{\small{Plot of $\log\log s_{k_N}(n)$ against $N$ for the first 24 primes in the sequence for $n=5$.}}
\label{nis5}
\end{figure}
 
\begin{figure} \centering
\includegraphics[width=10cm,height=10cm,keepaspectratio]{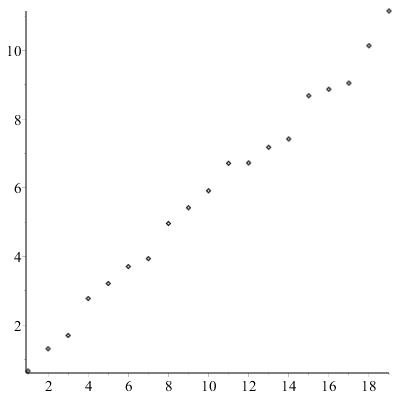}
\caption{\small{Plot of $\log\log s_{k_N}(n)$ against $N$ for the first 19 primes in the sequence for $n=6$.}}
\label{nis6}
\end{figure}

\begin{figure} \centering
\includegraphics[width=10cm,height=10cm,keepaspectratio]{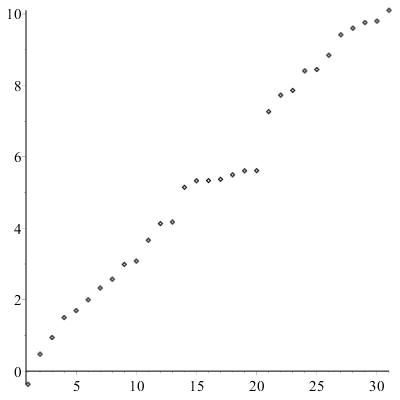}
\caption{\small{Plot of $\log\log r_{k_N}(n)$ against $N$ for the first 31 primes in the sequence for $n=3$.}}
\label{nism3}
\end{figure}

For $n=6$, the list of the first 25 
values $k$ for which $s_k(6)$ appear to be prime is  
$$ 
\begin{array}{l}

 1,
                               2,
                               3,
                               9,
                               14,
                               23,
                               29,
                               81,
                              128,
                              210,
                              468,
                              473,
                              746,
                              950,
                              3344,
                              4043,
                              4839, 
                             14376, \\
                             39521, 
64563, 72984, 82899, 84338, 85206, 86121,
\end{array}
$$ 
as given in OEIS sequence \seqnum{A113501},   
with the corresponding values of $s_k(6)$ given in  sequence \seqnum{A088165} 
(the prime NSW numbers \cite{nsw}). 
In our initial submission of this paper, we obtained the first 19 of these values independently, before we were aware of 
sequence A113501, and made the log log plot of these terms is as in Figure \ref{nis6}. The best fit line for these points has slope 
$$ 
C= 0.5434911190. 
$$
Subsequently we found the web page \cite{nswnumber}, where the last six indices above are listed separately, 
together with their date of discovery by Eric Weisstein. However, on that page it is stated 
unequivocally that all of the corresponding numbers $s_k(6)$ are prime, whereas presumably the largest of 
these values were obtained using Mathematica's probabilistic primality test, so the most that can be claimed is that 
they are probable primes. 

Assuming that the heuristic arguments given in section 6 above are correct, and that the small number of points plotted really gives an accurate picture 
of the behaviour for large $N$,  the predicted values for the ratio $\rho(n) = C/\log\sqrt{\la}$ in each case are 
$$ 
\rho(3)\approx 0.530689, \, \rho(4)\approx 0.789203, \, \rho(5)\approx 0.583174, \, \rho(6)\approx 0.616641. 
$$ 
Apart from the case $n=4$, all of these values are reasonably close to the number $e^{-\gamma}\approx 0.561459$ obtained from Mertens' theorem. The value for $n=4$ seems anomalous: there are fewer prime terms than predicted in this case. 
However, it may be unreasonable to expect close agreement with the predicted value, given the rather small number of data points plotted in each case.  

One can also consider the prime terms in the sequences 
$(\,r_k(n)\,)$, $n\geq 3$, corresponding to negative values of $n$ in $s_k(n)$. 
The list of the first 31  values $k$ for which $r_k(3)$ appear to be prime is  
$$ 
\begin{array}{l}
                               1,
                               2,
                               3,
                               5,
                               6,
                               8,
                               11,
                               14,
                               21,
                               23,
                               41,
                               65,
                               68,
                              179,
                              215,
                              216,
                              224,
                              254,
                              284,
                              285,
                              1485, \\
                              2361,
                              2693,
                              4655,  
                              4838,
                              7215,
                             12780,
                             15378,
                             17999,
                             18755,
                             25416.
\end{array}
$$ 
Figure \ref{nism3} is the log log plot of these terms. Note that 
$(\,r_k(3)\,)$ is a bisection of the Fibonacci sequence, for which 
the prime terms are isted as OEIS sequence \seqnum{A005478}. 
The slope of the best fit line for these points is  
$$ 
C= 
0.3409264905. 
$$
Dividing this value by $\log \sqrt{\la} = \log \left((1+\sqrt{5})/2\right)$ gives 
$$
\rho(-3)\approx 0.708475, 
$$ 
which is rather large compared with the value of  $e^{-\gamma}$ expected from Mertens' theorem, suggesting that the number of primes in this sequence is initially somewhat lower than would be 
expected from the heuristic argument in section 6. 

\section*{Appendix B: Related sequences  from the OEIS} 

Here we briefly mention some other sequences in the OEIS which are related to the 
considerations in this paper. 

Sequence \seqnum{A294099} contains the array of values $s_k(n)$ for $n\geq 1$, $k\geq 0$, 
while A299045 is the array of $s_k(-n)$ for the same range of $n$ and $k$.

Sequence \seqnum{A002327} consists of primes of the form $n^2-n-1$, and after sending $n\to -n$ this corresponds to prime values of the polynomial $s_2(n)=n^2+n-1$, for which the relevant values of $n$ are given by 
sequence \seqnum{A045546}. 

Sequence \seqnum{A000032} begins 
$$ 
2,1,3,4,7,11,18,29,47,\ldots,  
$$ 
and consists of the Lucas numbers denoted $\ell^+_k(1,-1)$ in section 2, which 
satisfy the Fibonacci  recurrence $\ell^+_{k+2}(1,-1)=\ell^+_{k+1}(1,-1)+\ell^+_k(1,-1)$. 
This  coincides with an interlacing of two sequences, namely 
$$ 
\Tc_0(3), s_0(3), \Tc_1(3), s_1(3), \Tc_2(3), s_2(3),\Tc_3(3), \ldots , 
$$ 
so its two distinct bisections are $(\, \Tc_k(3)\,)$ and  $(\, s_k(3)\,)$, given by \seqnum{A005248} and \seqnum{A002878} respectively.  
Similarly, the Fibonacci sequence  \seqnum{A000045} itself  coincides with the interlacing  
$$ 
\Uc_{-1}(3), r_0(3), \Uc_0(3), r_1(3), \Uc_1(3), r_2(3),\Uc_2(3), \ldots 
$$ 
obtained from $(\, \Uc_k(3)\,)$ and  $(\, r_k(3)\,)$, given by \seqnum{A001906} and \seqnum{A001519} respectively.  

There are other values of $n$ for which the OEIS entry for the sequence of terms $s_k(n)$ has not been mentioned so far: $(\,s_k(4)\,)_{k\geq 0}$ is \seqnum{A001834}, $(\,s_k(5)\,)_{k\geq 0}$ is \seqnum{A030221}, 
 $(\,s_k(7)\,)_{k\geq 0}$ is \seqnum{A033890},  $(\,s_k(8)\,)_{k\geq 0}$ is \seqnum{A057080},  and $(\,s_k(9)\,)_{k\geq 0}$ is \seqnum{A057081}.

\seqnum{A008865} is the sequence of values of $\Tc_2(j)$ for $j=1,2,3,\ldots$; the 
array of values $\Tc_k(n)$ for $k\geq 1$, $n\geq 1$ is rendered as sequence 
\seqnum{A298675}.  The values $\Tc_p(n)$ for prime $p$ are listed in sequence \seqnum{A298878}, while the 
values   $\Tc_p(n)$ with $p$ an odd prime which are not also of the form $\Tc_2(m)$ for some $m$ are 
given in \seqnum{A299071}.

\end{document}